\documentclass[final,a4paper,11pt]{article}
\usepackage{graphicx}
\usepackage{mathptmx}
\usepackage{amssymb,amsmath, amsfonts,amsthm}
\usepackage{a4wide}
\usepackage{color}

\usepackage{url}
\usepackage{tikz}
\usepackage{pgfplots}
\pgfplotsset{compat=1.14}
\usepackage{comment}
\usepackage{cases}
\usepackage{float}
\usepackage{array}
\usepackage{tabularx}
\usepackage{ragged2e}
\usepackage{comment}

\newcommand{\email}[1]{{\tt #1}}
\newcommand{\R}{\mathbb{R}}

\newcommand{\norm}[1]{\|#1\|}

\newcommand{\dist}[1]{{\rm dist}(#1)}

\newcommand{\mv}{\,\mid\,}

\newcommand{\A}{{\cal A}}
\newcommand{\B}{{\cal B}}

\newcommand{\I}{{\cal I}}

\newcommand{\M}{{\cal M}}
\newcommand{\New}{{\cal N}}

\newcommand{\T}{{\cal T}}
\newcommand{\V}{{\cal V}}
\newcommand{\U}{{\cal U}}
\newcommand{\Sp}{{\mathcal S}}
\newcommand{\F}{{\cal F}}
\newcommand{\Z}{{\cal Z}}

\newcommand{\longsetto}[1]{\mathop{\longrightarrow}\limits^{#1}}
\newcommand{\skalp}[1]{\langle #1\rangle}

\newcommand{\xb}{\bar x}
\newcommand{\yb}{\bar y}

\newcommand{\AT}[2]{{\textstyle{#1\atop#2}}}
\newcommand{\xba}{{\bar x^\ast}}

\newcommand{\OO}{{\cal O}}
\newcommand{\argmin}{\mathop{\rm arg\,min}}

\newcommand{\Span}{{\rm span\,}}

\newcommand{\ri}{{\rm ri\,}}

\newcommand{\gph}{\mathrm{gph}\,}
\newcommand{\dom}{\mathrm{dom}\,}

\newcommand{\tto}{\rightrightarrows}
\newcommand{\Limsup}{\mathop{{\rm Lim}\,{\rm sup}}}
\newcommand{\myvec}[1]{\begin{pmatrix}#1\end{pmatrix}}
\newcommand{\SCD}{SCD\ }
\newcommand{\subreg}{{\rm subreg\,}}
\newcommand{\lsubreg}{{\rm l\mbox{-}subreg\,}}

\newcommand{\reg}{{\rm reg\,}}
\newcommand{\ssstar}{semismooth$^{*}$ }
\newcommand{\ee}[2]{#1^{(#2)}}

\newcommand{\rge}{{\rm rge\;}}
\renewcommand{\ker}{{\rm ker\;}}

\newlength{\myAlgBox}
\newtheorem{theorem}{Theorem}[section]
\newtheorem{proposition}[theorem]{Proposition}
\newtheorem{remark}[theorem]{Remark}
\newtheorem{lemma}[theorem]{Lemma}

\newtheorem{definition}[theorem]{Definition}
\newtheorem{example}[theorem]{Example}
\newtheorem{algorithm}{Algorithm}

\title{On the application of the \SCD semismooth* Newton method \\ to
variational inequalities of the second kind}
\author{Helmut Gfrerer\thanks{Institute of Computational Mathematics, Johannes Kepler University
Linz, A-4040 Linz, Austria; \email{helmut.gfrerer@jku.at}}
 \and   Ji\v{r}\'{i} V. Outrata\thanks{Institute of Information Theory and Automation, Czech Academy of
 Sciences, 18208 Prague, Czech Republic, and Centre for
              Informatics and Applied Optimization, Federation University of Australia, POB 663,
              Ballarat,  Vic 3350, Australia,  \email{outrata@utia.cas.cz}}
\and{Jan Valdman\thanks{Institute of Information Theory and Automation, Czech Academy of
 Sciences, 18208 Prague, Czech Republic, and Institute of Mathematics, Faculty of Science,
University of South Bohemia,
37005~\v{C}esk\'{e}~Bud\v{e}jovice,
Czech Republic, \email{jan.valdman@utia.cas.cz}}}}
\date{}

\begin{document}
\maketitle

{\footnotesize
{\bf Abstract.}
The paper starts with a description of the SCD (subspace containing derivative) mappings and the SCD \ssstar Newton method for the solution of general inclusions. This method is then applied to a class of  variational inequalities of the second kind. As a result, one obtains an implementable algorithm exhibiting a locally superlinear convergence. Thereafter we suggest several globally convergent hybrid algorithms in which one combines the SCD  \ssstar Newton method with selected splitting algorithms for the solution of monotone variational inequalities. Finally we demonstrate the efficiency of one of these methods via a Cournot-Nash equilibrium, modeled as a variational inequalities of the second kind, where one admits really large numbers of players (firms) and produced commodities.
}

\section{Introduction}
In \cite{GfrOut21} the authors proposed the so-called semismooth* Newton method for the numerical solution of a general inclusion
\begin{equation*}
0 \in H(x),
\end{equation*}
where $H:\mathbb{R}^{n} \rightrightarrows   \mathbb{R^{n}}$ is a closed-graph multifunction. This method has been further developed in \cite{GfrOut21a}, where it has coined the name SCD (subspace containing derivative) semismooth* Newton method. When compared with the original method from \cite{GfrOut21a}, the new variant requires a slightly stronger approximation of the limiting coderivative of $H$, but exhibits  locally superlinear convergence under substantially less restrictive assumptions. The aim of this paper is to work out this Newton-type method for the numerical solution of the {\em generalized equation} (GE)
\begin{equation}\label{EqVI2ndK}
0 \in H(x):= f(x)+\partial q (x),
\end{equation}
where $f:\mathbb{R}^{n} \rightarrow \mathbb{R}^{n}$ is continuously differentiable, $q:\mathbb{R}^{n} \rightarrow \overline{\mathbb{R}}$ is proper convex and lower-semicontinuous (lsc) and $\partial$ stands for the classical Moreau-Rockafellar subdifferential. It is easy to see that GE (\ref{EqVI2ndK}) is equivalent to the variational inequality  (VI):

Find $\bar{x} \in \mathbb{R}^{n}$ such that
\begin{equation}\label{eq-102}
\langle f(\xb), x-\bar{x}\rangle + q(x)-q(\bar{x})\geq 0 ~ \mbox{ for all } ~ x \in \mathbb{R}^{n}.
\end{equation}

The model \eqref{eq-102} has been introduced in \cite{Gl84} and one speaks about the {\em variational inequality (VI) of the second kind}. It is widely used in the literature dealing with equilibrium models in continuum mechanics cf., e.g., \cite{Hasl} and the references therein. For the numerical solution of GE \eqref{EqVI2ndK}, a number of methods can be used ranging from nonsmooth optimization methods (applicable when $\nabla f$ is symmetric) up to a broad family of splitting methods (usable when $H$ is monotone), cf. \cite[Chapter 12]{FaPa03}. If GE \eqref{EqVI2ndK} amounts to stationarity condition for a Nash game, then also a simple coordinate-wise optimization technique can be used, cf. \cite{KaSch18} and \cite{OV}. Concerning the Newton type methods, let us mention, for instance, the possibility to write down GE \eqref{EqVI2ndK} as an equation on a monotone graph, which enables
us to apply the Newton procedure from \cite{Rob11}. Note, however, that the subproblems to be solved in this approach are typically rather difficult. In other papers the authors reformulate the problem as a (standard) nonsmooth equation which is then solved by the classical semismooth Newton method, see, e.g., \cite{ItKu09, XiLiWeZh18}.

 As mentioned above, in this paper we will investigate the numerical solution of GE (\ref{EqVI2ndK}) via the SCD semismooth* Newton method developed in \cite{GfrOut21a}. In contrast to the Newton methods by Josephy, in this method (as well as in its original variant from \cite{GfrOut21})  the multi-valued part of (\ref{EqVI2ndK}) is also approximated and, differently to some other Newton-type methods, this approximation is provided by means of a linear subspace belonging to the graph of the limiting coderivative of $\partial q$.  In this way the computation of the Newton direction
 reduces to the solution of a linear system of equations. To ensure locally superlinear convergence, two properties have to be fulfilled. The first one is a weakening of the \ssstar property from \cite{GfrOut21} and pertains the subdifferential mapping
 $\partial q$. The second one, called
 SCD regularity, concerns the mapping $H$ and
 amounts, roughly speaking, to the strong
 metric subregularity of the considered GE {\em around} the solution.

 The plan of the paper is as follows. After the preliminary Section 2, where we provide the needed background from modern variational analysis, Section 3 is devoted
 to the broad class of SCD mappings, which is the basic framework for the application of
 the used method. In particular,
  the subdifferential of a
 proper convex lsc function is an SCD mapping. In Section 4 the SCD semismooth* Newton method is described and its convergence is analyzed. Thereafter, in Section 5 we develop an implementable version of the method for the solution of GE \eqref{EqVI2ndK} and show its locally superlinear convergence under mild assumptions. Section 6 deals with the issue of global convergence. First we suggest a heuristic modification of the method from the preceding section which exhibits very good convergence properties in the numerical experiments. Thereafter we show global convergence for a family of hybrid algorithms, where one combines the semismooth* Newton method with various frequently used splitting methods.
 Finally, in Section 7  we demonstrate the efficiency of the developed methods via a Cournot-Nash equilibrium problem taken over from \cite{MSS} which can be modeled in the form of GE \eqref{EqVI2ndK}. In contrast to the numerical approach in \cite{MSS}, we may work here with "arbitrarily" large numbers of player (firms) and commodities,

  The following notation is employed. Given a matrix $A$,  $\rge A$ and $\ker A$ denote the range space and the kernel of $A$, respectively, and $\norm{A}$ stands for its spectral norm. For a set $\Omega$, $\dist{x,\Omega}:=\inf_{a \in \Omega} \norm{x-a}$
  signifies the
  distance from $x$ to $\Omega$ and  $\ri \Omega$ is the relative interior of $\Omega$. Further, $L^\perp$ denotes the anihilator of a linear subspace $L$ and ${\rm diag\,} (A,B)$ means a block diagonal matrix with matrices $A, B$ as diagonal blocks.

\section{Preliminaries}
Throughout the whole paper, we will frequently use  the following basic notions of modern
variational analysis.

 \begin{definition}
 Let $A$  be a  set in $\mathbb{R}^{s}$, $\bar{x} \in A$ and $A$ be locally closed around $\bar{x}$. Then
\begin{enumerate}
 \item [(i)]The  {\em tangent (contingent, Bouligand) cone}
 to $A$ at $\bar{x}$ is given by
 \[
 T_{A}(\bar{x}):=\Limsup\limits_{t\downarrow 0} \frac{A-\bar{x}}{t}.
 \]
 \item[(ii)] The set
 \[\widehat{N}_{A}(\bar{x}):=(T_{A}(\bar{x}))^{\circ}\]
 is the {\em regular (Fr\'{e}chet) normal cone} to $A$ at $\bar{x}$, and
 \[N_{A}(\bar{x}):=\Limsup\limits_{\stackrel{A}{x \rightarrow \bar{x}}} \widehat{N}_{A}(x)\]
 is the {\em limiting (Mordukhovich) normal cone} to $A$ at $\bar{x}$. \if{Given a direction $d
 \in\mathbb{R}^{s}$,
\[ N_{A}(\bar{x};d):= \Limsup\limits_{\stackrel{t\downarrow 0}{d^{\prime}\rightarrow
 d}}\widehat{N}_{A}(\bar{x}+ td^{\prime})\]
 is the {\em directional limiting normal cone} to $A$ at $\bar{x}$ {\em in direction} $d$.}\fi
 \end{enumerate}
\end{definition}
In this definition ''$\Limsup$'' stands for the Painlev\' e-Kuratowski {\em outer (upper) set limit}, see, e.g., \cite{RoWe98}.
The above listed cones enable us to describe the local behavior of set-valued maps via various
generalized derivatives. Let $F:\R^n\tto\R^m$ be a (set-valued) mapping with the domain and the graph
\[\dom F:=\{x\in\R^n\mv F(x) \not = \emptyset\},\quad \gph F:=\{(x,y)\in\R^n\times\R^m \mv y\in F(x)\}.\]

\begin{definition}
Consider a (set-valued) mapping $F:\R^n\tto\R^m$ and let $\gph F$ be locally closed around some $(\xb,\yb)\in \gph F$.
\begin{enumerate}
\item[(i)]
 The multifunction $DF(\xb,\yb):\R^n\tto\R^m$, given by $\gph DF(\xb,\yb)=T_{\gph F}(\xb,\yb)$, is called the {\em graphical derivative} of $F$ at $(\xb,\yb)$.
\item [(ii)]  The multifunction $D^\ast F(\xb,\yb ): \R^m \tto \R^n$,  defined by
 \[ \gph D^\ast F(\xb,\yb )=\{(y^*,x^*)\mv (x^*,-y^*)\in N_{\gph F}(\xb,\yb)\}\]
is called the {\em limiting (Mordukhovich) coderivative} of $F$ at $(\xb,\yb )$.
\end{enumerate}
\end{definition}

Let us now recall the following regularity notions.
\begin{definition}
  Let $F:\R^n\tto\R^m$ be a a (set-valued) mapping and let $(\xb,\yb)\in\gph F$.
  \begin{enumerate}
    \item $F$ is said to be {\em metrically subregular at} $(\xb,\yb)$ if there exists $\kappa\geq 0$  along with some  neighborhood $X$ of $\xb$ such that
    \begin{equation*}
      \dist{x,F^{-1}(\yb)}\leq \kappa\,\dist{\yb,F(x)}\ \forall x\in X.
    \end{equation*}
  \item $F$ is said to be {\em strongly metrically subregular at} $(\xb,\yb)$ if it is metrically subregular at $(\xb,\yb)$ and there exists a neighborhood $X'$ of $\xb$ such that $F^{-1}(\yb)\cap X'=\{\xb\}$.
  \item $F$ is said to be {\em metrically regular around} $(\xb,\yb)$ if there is $\kappa\geq 0$ together with neighborhoods $X$ of $\xb$ and $Y$ of $\yb$ such that
      \begin{equation*}
      \dist{x,F^{-1}(y)}\leq \kappa\,\dist{y,F(x)}\ \forall (x,y)\in X\times Y.
    \end{equation*}
  \item $F$ is said to be {\em strongly metrically regular around} $(\xb,\yb)$ if it is metrically regular around $(\xb,\yb)$ and $F^{-1}$ has a single-valued localization around $(\yb,\yb)$, i.e., there are  open neighborhoods $Y'$ of $\yb$, $X'$ of $\xb$ and a mapping $h:Y'\to\R^n$ with $h(\yb)=\xb$ such that $\gph F\cap (X'\times Y')=\{(h(y),y)\mv y\in Y'\}$.
  \end{enumerate}
\end{definition}

It is easy to see that the strong metric regularity around $(\xb,\yb)$ implies the strong metric subregularity  at $(\xb,\yb)$ and the metric regularity around $(\xb,\yb)$ implies the metric subregularity at $(\xb,\yb)$. To check the metric regularity one often employs the so-called
Mordukhovich criterion, according to which this property around $(\xb,\yb)$
is equivalent with the condition
\begin{equation}\label{EqMoCrit}
 0\in D^*F(\xb,\yb)(y^*)\ \Rightarrow\ y^*=0.
\end{equation}
For pointwise characterizations of the other stability properties from Definition 2.3 the reader is referred to \cite[Theorem 2.7]{GfrOut21a}.

We end up this preparatory section with a definition of the \ssstar property which paved the way both to \ssstar Newton method in \cite{GfrOut21} as well as to the SCD \ssstar Newton method in \cite{GfrOut21a}.
\begin{definition}
We say that $F:\R^n\tto\R^n$ is {\em \ssstar}  at $(\xb,\yb)\in\gph F$ if
for every $\epsilon>0$ there is some $\delta>0$ such that the inequality
\begin{align*}
\vert \skalp{x^*,x-\xb}-\skalp{y^*,y-\yb}\vert&\leq \epsilon
\norm{(x,y)-(\xb,\yb)}\norm{(x^*,y^*)}
\end{align*}
holds for all $(x,y)\in \gph F\cap \B_\delta(\xb,\yb)$ and all $(y^*,x^*)$ belonging to $\gph D^\ast F(\xb,\yb )$.
\end{definition}

\if{
It is well-known, see, e.g. \cite{DoRo14}, that the property of (strong) metric subregularity for $F$ at $(\xb,\yb)$ is equivalent with the property of {\em (isolated) calmness}  for $F^{-1}$ at $(\yb,\xb)$.
Further, $F$ is metrically regular around $(\xb,\yb)$ if the inverse mapping $F^{-1}$ has the so-called {\em Aubin property} around $(\yb,\xb)$. In this paper we will frequently use the following characterization of strong metric regularity.
\begin{theorem}[{ cf. \cite[Proposition 3G.1]{DoRo14}}]\label{ThStrMetrReg} $F:\R^n\tto\R^m$ is strongly metrically regular around $(\xb,\yb)$ if and only if $F^{-1}$ has a Lipschitz continuous localization $h$  around $(\yb,\xb)$. In this case there holds
\[\reg F(\xb,\yb)=\limsup_{\AT{y,y'\to\yb}{y\not=y'}}\frac {\norm{h(y)-h(y')}}{\norm{y-y'}}.\]
\end{theorem}

In this paper we will also use the  following point-based characterizations of the above regularity properties.
\begin{theorem}\label{ThCharRegByDer}
    Let $F:\R^n\tto\R^m$ be a mapping and let $(\xb,\yb)\in\gph F$.
    \begin{enumerate}
    \item[(i)] (Levy-Rockafellar criterion) $F$ is strongly metrically subregular at $(\xb,\yb)$ if and only if
      \begin{equation}\label{EqLevRoCrit}
        0\in DF(\xb,\yb)(u)\ \Rightarrow u=0.
      \end{equation}
      and in this case one has
      \[\subreg F(\xb,\yb)=\sup\{\norm{u}\mv\exists v: (u,v)\in\gph DF(\xb,\yb),\ \norm{v}\leq 1\}.\]
    \item[(ii)] (Mordukhovich criterion) $F$ is metrically regular around $(\xb,\yb)$ if and only if
      \begin{equation}
            \label{EqMoCrit} 0\in D^*F(\xb,\yb)(y^*)\ \Rightarrow\ y^*=0.
      \end{equation}
      Further, in this case one has
      \begin{equation}
        \label{EqModMetrReg} {\rm reg\;}F(\xb,\yb)=\sup\{\norm{y^*}\mv\exists x^*: (y^*,x^*)\in\gph D^*F(\xb,\yb),\ \norm{x^*}\leq 1\}.
      \end{equation}
    \item[(iii)] $F$ is strongly metrically regular around $(\xb,\yb)$ if and only if
      \begin{equation}
            \label{EqStrictCrit} 0\in D_*F(\xb,\yb)(u)\ \Rightarrow\ u=0
      \end{equation}
      and \eqref{EqMoCrit} holds. In this case one also has
      \begin{equation}\label{EqModStrMetrReg}
        {\rm reg\;}F(\xb,\yb) = \sup\{\norm{u}\mv\exists v:\ (u,v)\in\gph D_*F(\xb,\yb),\ \norm{v}\leq 1\}.
      \end{equation}
    \end{enumerate}
\end{theorem}
\begin{proof}
  Statement (i)  follows from \cite[Theorem 4E.1]{DoRo14}. Statement (ii) can be found in \cite[Theorem 3.3]{Mo18}. The criterion for strong metric regularity follows from Dontchev and Frankowska \cite[Theorem 16.2]{DoFra13} by taking into account that the condition $\xb\in \liminf_{y\to\yb}F^{-1}(y)$  appearing in \cite[Theorem 16.2]{DoFra13} can be ensured by the requirement that $F$ is metrically regular which in turn can be characterized by the Mordukhovich criterion.
\end{proof}
For a sufficient condition for metric subregularity based on directional limiting coderivatives we refer to \cite{GO3}.

The properties of (strong) metric regularity and strong metric subregularity are stable under Lipschitzian and calm  perturbations, respectively, cf. \cite{DoRo14}. Further note that the property of (strong) metric regularity hold around all points belonging to the graph of $F$ sufficiently close to the reference point, whereas the property of (strong) metric subregularity
is guaranteed to hold only at the reference point. This leads to the following definition.
\begin{definition}\label{DefLocStrSubReg}
  We say that the mapping $F:\R^n\tto\R^m$ is {\em (strongly) metrically subregular around} $(\xb,\yb)\in\gph F$ if there is a neighborhood $W$ of $(\xb,\yb)$ such that $F$ is (strongly) metrically subregular at every point $(x,y)\in\gph F\cap W$ and we define
  \[\lsubreg F(\xb,\yb):=\limsup_{(x,y)\longsetto{{\gph F}}(\xb,\yb)}\subreg F(x,y)<\infty.\]
  In this case we will also speak about {\em (strong) metric subregularity on a neighborhood}.
\end{definition}
Note that every polyhedral multifunction, i.e., a mapping whose graph is the union of finitely many convex polyhedral sets, is metrically subregular around every point of its graph by Robinson's result \cite{Ro81}. In Section \ref{SecStrMetrSubr}, characterizations of strong metric subregularity on a neighborhood will be investigated.

Next we  introduce the \ssstar sets and mappings.

\begin{definition}\label{DefSemiSmooth}(cf. \cite{GfrOut21}.)
\begin{enumerate}
\item  A set $A\subseteq\R^s$ is called {\em \ssstar} at a point $\xb\in A$ if for all $u\in
    \R^s$ it holds
  \begin{equation}\label{EqSemiSmoothSet}\skalp{x^*,u}=0\quad \forall x^*\in N_A(\xb;u).
  \end{equation}
\item
A set-valued mapping $F:\R^n\tto\R^m$ is called {\em \ssstar} at a point $(\xb,\yb)\in\gph F$, if
$\gph F$ is \ssstar at $(\xb,\yb)$, i.e., for all $(u,v)\in\R^n\times\R^m$ we have
\begin{equation}\label{EqSemiSmooth}
\skalp{u^*,u}=\skalp{v^*,v}\quad \forall (v^*,u^*)\in\gph D^*F((\xb,\yb);(u,v)).
\end{equation}
\end{enumerate}
\end{definition}
The class of semismooth* mappings is rather broad.
We list here two important classes of multifunctions having this property.
\begin{proposition}
  \label{PropSSstar}
  \begin{enumerate}
    \item[(i)]Every mapping whose graph is the union of finitely many closed convex sets is \ssstar at every point of its graph.
    \item[(ii)]Every mapping with closed subanalytic graph is \ssstar at every point of its graph.
  \end{enumerate}
\end{proposition}
\begin{proof}
The first assertion was already shown in \cite[Proposition 3.4, 3.5]{GfrOut21}. As mentioned in \cite[Remark 3.10]{GfrOut21}, the \ssstar property of sets amounts to the notion of semismoothness introduced in \cite{HO01}. It follows thus from \cite[Theorem 2]{Jou07},  that all closed subanalytic sets  are automatically \ssstar and the second statement holds by the definition of \ssstar mappings.
\end{proof}
The statement of Proposition \ref{PropSSstar}(ii) can be considered as the counterpart to \cite{BoDaLe09}, where it is shown that locally Lipschitz tame mappings $F:U\subseteq\R^n\to\R^m$ are semismooth in the sense of Qi and Sun \cite{QiSun93}.
In case of single-valued Lipschitzian mappings the \ssstar property is equivalent with the semismooth property introduced by Gowda \cite{Gow04}, which is weaker than the one in \cite{QiSun93}.

In the above definition the   \ssstar sets and mappings have been defined via directional limiting normal
cones and coderivatives. For our purpose it is convenient to make use of equivalent
characterizations in terms of standard (regular and limiting) normal cones and coderivatives,
respectively.

\begin{proposition}[{cf.\cite[Corollary 3.3]{GfrOut21}}]\label{PropCharSemiSmooth}Let $F:\R^n\tto\R^m$ and $(\xb,\yb)\in \gph F$ be given.
Then the following three statements are equivalent
\begin{enumerate}
\item[(i)] $F$ is \ssstar at $(\xb,\yb)$.
\item[(ii)] For every $\epsilon>0$ there is some $\delta>0$ such that
\begin{equation}\label{EqCharSemiSmoothReg}
\hspace{-1cm}\vert \skalp{x^*,x-\xb}-\skalp{y^*,y-\yb}\vert\leq \epsilon
\norm{(x,y)-(\xb,\yb)}\norm{(x^*,y^*)}\ \forall(x,y)\in \B_\delta(\xb,\yb)\ \forall
(y^*,x^*)\in\gph \widehat D^*F(x,y).\end{equation}
\item[(iii)] For every $\epsilon>0$ there is some $\delta>0$ such that
  \begin{equation}\label{EqCharSemiSmoothLim}
\hspace{-1cm}\vert \skalp{x^*,x-\xb}-\skalp{y^*,y-\yb}\vert\leq \epsilon
\norm{(x,y)-(\xb,\yb)}\norm{(x^*,y^*)}\ \forall(x,y)\in \B_\delta(\xb,\yb)\ \forall
(y^*,x^*)\in\gph D^*F(x,y).
\end{equation}
\end{enumerate}
\end{proposition}
}\fi

\section{On \SCD mappings}
\subsection{Basic properties}
In this section we want to recall the basic definitions and features  of the \SCD property  introduced in the recent paper \cite{GfrOut21a}.

In what follows we denote by $\Z_n$ the metric space of all $n$-dimensional subspaces of $\R^{2n}$ equipped with the metric
\[d_\Z(L_1,L_2):=\norm{P_{L_1}-P_{L_2}}\]
where $P_{L_i}$ is the symmetric $2n\times 2n$ matrix representing the orthogonal projection on $L_i$, $i=1,2$.

Sometimes we will also work with bases for the subspaces $L\in\Z_n$. Let $\M_n$ denote the collection of all $2n\times n$ matrices with full column rank $n$ and for $L\in \Z_n$ we define
\[\M(L):=\{Z\in \M_n\mv \rge Z =L\},\]
i.e., the columns of $Z\in\M(L)$ are a basis for $L$.

We treat every element of $\R^{2n}$ as a column vector. In order to keep notation simple we write $(u,v)$ instead of $\myvec{u\\v}\in\R^{2n}$ when this does not lead to confusion. In order to refer to the components of the vector $z=\myvec{u\\ v}$ we set $\pi_1(z):=u,\ \pi_2(z)=v$.

Let $L\in\Z_n$ and consider $Z\in \M(L)$. Then we can partition $Z$ into two $n\times n$ matrices $A$ and $B$ and we will write $Z=(A,B)$ instead of $Z=\myvec{A\\B}$. It follows that $\rge(A,B):=\{(Au,Bu)\mv u\in\R^n\}\doteq \{\myvec{Au\\Bu}\mv u\in\R^n\}=L$. Similarly as before, we will also use $\pi_1(Z):=A$, $\pi_2(Z):=B$ for referring to the two $n\times n$ parts of $Z$.

Further, for every $L\in \Z_n$ we can define
\begin{align*}
  L^*&:=\{(-v^*,u^*)\mv (u^*,v^*)\in L^\perp\},
\end{align*}
where $L^\perp$ denotes as usual the orthogonal complement of $L$. Then it can be shown that $(L^*)^*=L$ and $d_\Z(L_1,L_2)=d_\Z(L_1^*,L_2^*)$. Thus the mapping $L\to L^*$ defines an isometry on $\Z_n$.

We denote by $S_n$ the $2n\times 2n$ orthogonal matrix
\[S_n:=\begin{pmatrix}0&-I\\I&0\end{pmatrix}\]
so that $L^*=S_nL^\perp$.

\begin{definition}\label{DefSCDProperty}
  Consider a mapping  $F:R^n\tto\R^n$.
  \begin{enumerate}
    \item We call $F$ {\em graphically smooth of dimension $n$} at $(x,y)\in \gph F$, if $T_{\gph F}(x,y)=\gph DF(x,y)\in \Z_n$. Further we denote by $\OO_F$ the set of all points where $F$ is graphically smooth of dimension $n$.
    \item We associate with $F$ the four mappings $\widehat\Sp F:\gph F\tto \Z_n$, $\widehat\Sp^* F:\gph F\tto \Z_n$, $\Sp F:\gph F\tto \Z_n$, $\Sp^* F:\gph F\tto \Z_n$, given by
    \begin{align*}\widehat\Sp F(x,y)&:=\begin{cases}\{\gph DF(x,y)\}& \mbox{if $(x,y)\in\OO_F$,}\\
    \emptyset&\mbox{else,}\end{cases}\\
    \widehat\Sp^* F(x,y)&:=\begin{cases}\{\gph DF(x,y)^*\}& \mbox{if $(x,y)\in\OO_F$,}\\
    \emptyset&\mbox{else,}\end{cases}\\
    \Sp F(x,y)&:=\Limsup_{(u,v)\longsetto{{\gph F}}(x,y)} \widehat\Sp F(u,v) \\
    &=\{L\in \Z_n\mv \exists (x_k,y_k)\longsetto{{\OO_F}}(x,y):\ \lim_{k\to\infty} d_\Z(L,\gph DF(x_k,y_k))=0\},\\
    \Sp^* F(x,y)&:=\Limsup_{(u,v)\longsetto{{\gph F}}(x,y)} \widehat\Sp^* F(u,v)\\
    &=\{L\in \Z_n\mv \exists (x_k,y_k)\longsetto{{\OO_F}}(x,y):\ \lim_{k\to\infty} d_\Z(L,\gph DF(x_k,y_k)^*)=0\}.
    \end{align*}
    \item We say that $F$ has the {\em\SCD (subspace containing derivative) property at} $(x,y)\in\gph F$, if $\Sp^*F(x,y)\not=\emptyset$. We say that $F$ has the \SCD property {\em around} $(x,y)\in\gph F$, if there is a neighborhood $W$ of $(x,y)$ such that $F$ has the \SCD property at every $(x',y')\in\gph F\cap W$.
     Finally, we call $F$ an {\em \SCD mapping} if
    $F$ has the \SCD property at every point of its graph.
  \end{enumerate}
\end{definition}
Since $L\to L^*$ is an isometry on $\Z_n$ and $(L^*)^*=L$, the mappings $\Sp^*F$ and $\Sp F$ are related via
  \[\Sp^* F(x,y)=\{L^*\mv L\in \Sp F(x,y)\},\ \Sp F(x,y)=\{L^*\mv L\in \Sp^* F(x,y)\}.\]
The name \SCD property is motivated by the following statement.
\begin{lemma}[cf.{\cite[Lemma 3.7]{GfrOut21a}}] Let $F:\R^n\tto\R^n$ and let $(x,y)\in \gph F$. Then $L\subseteq \gph D^*F(x,y)$ $\forall L\in \Sp^*F(x,y)$.
\end{lemma}
Next we turn to the notion of \SCD regularity.
\begin{definition}
\begin{enumerate}
\item We denote by $\Z_n^{\rm reg}$ the collection of all subspaces $L\in\Z_n$ such that
  \begin{equation*}
    (y^*,0)\in L\ \Rightarrow\ y^*=0.
  \end{equation*}
  \item  A mapping $F:\R^n\tto\R^n$ is called {\em \SCD regular around} $(x,y)\in\gph F$, if $F$ has the \SCD property around $(x,y)$ and
  \begin{equation}\label{EqSCDReg}
    (y^*,0)\in L \Rightarrow\ y^*=0\ \forall L\in \Sp^*F(x,y),
  \end{equation}
  i.e., $L\in \Z_n^{\rm reg}$ for all $L\in \Sp^*F(x,y)$. Further, we will denote by
  \[{\rm scd\,reg\;}F(x,y):=\sup\{\norm{y^*}\mv (y^*,x^*)\in L, L\in \Sp^*F(x,y), \norm{x^*}\leq 1\}\] 
  the {\em modulus of \SCD regularity} of $F$ around $(x,y)$.
\end{enumerate}
\end{definition}
Since the elements of $\Sp^*F(x,y)$ are contained in $\gph D^*F(x,y)$, it follows from the Mordukhovich criterion \eqref{EqMoCrit} that \SCD regularity is weaker than metric regularity.

In the following propostion we state some basic properties of subspaces $L\in\Z_n^{\rm reg}$.
\begin{proposition}[cf.{\cite[Proposition 4.2]{GfrOut21a}}] \label{PropC_L}
    Given a $2n\times n$ matrix $Z$, there holds $\rge Z\in \Z_n^{\rm reg}$ if and only if the $n\times n$ matrix $\pi_2(Z)$ is nonsingular. Thus,
    for every $L\in \Z_n^{\rm reg}$   there is a  unique $n\times n$ matrix $C_L$  such that $L=\rge(C_L,I)$. Further, $L^*=\rge(C_L^T,I)\in\Z_n^{\rm reg}$,
    \begin{equation*}
    \skalp{x^*,C_L^Tv}=\skalp{y^*,v}\ \forall (y^*,x^*)\in L\forall v\in\R^n.
    \end{equation*}
    and
    \begin{equation*}
    \norm{y^*}\leq \norm{C_L}\norm{x^*}\ \forall (y^*,x^*)\in L.
  \end{equation*}
\end{proposition}
Note that for every $L\in\Z_n^{\reg}$ and every $(A,B)\in \M(L)$ the matrix $B$ is nonsingular and $C_L=AB^{-1}$.

Combining \cite[Equation  (34), Lemma 4.7, Proposition 4.8]{GfrOut21a} we obtain the following lemma
\begin{lemma}\label{LemSCDReg}
   Assume that $F:\R^n\tto\R^n$ is \SCD regular around $(\xb,\yb)\in\gph F$. Then
   \[{\rm scd\,reg\;}F(\xb,\yb)=\sup\{\norm{C_L}\mv L\in\Sp^*F(\xb,\yb)\}<\infty\]
   Moreover, $F$ is \SCD regular around every $(x,y)\in\gph F$ sufficiently close to $(\xb,\yb)$ and
  \[\limsup_{(x,y)\longsetto{\gph F}(\xb,\yb)}{\rm scd\,reg\;}F(x,y)\leq{\rm scd\,reg\;}F(\xb,\yb).\]
\end{lemma}

\subsection{
On the \SCD property of the subdifferential  of convex functions}

\begin{theorem}[cf.{\cite[Corollary 3.28]{GfrOut21a}}]\label{ThSubdiffConv}
  For every proper lsc convex function $q:\R^n\to\bar\R$ the subdifferential mapping $\partial q$ is   an \SCD mapping and for every $(x,x^*)\in \gph\partial q$ and for every $L\in \Sp^*\partial q(x,x^*)=\Sp\partial q(x,x^*)$ there is a symmetric positive semidefinite $n\times n$ matrix $B$ with $\norm{B}\leq 1$ such that $L=\rge(B,I-B)=L^*$.
\end{theorem}
The representation of $L$ via the matrix $B$ is only one possibility. E.g., if $q$ is twice continuously differentiable then $\rge(I,\nabla^2 q(x))=\gph D^*\partial q(x,\nabla q(x))$ and the relation between $B$ and $\nabla^2q(x)$ is given by $B=(I+\nabla^2 q(x))^{-1}$, $I-B=(I+\nabla^2 q(x))^{-1}\nabla^2 q(x)$ and $\nabla^2 q(x)=B^{-1}(I-B)$.
\begin{example}
  Assume that $q(x)=\norm{x}$ so that
  \[\partial q(x)=\begin{cases}\B&\mbox{for $x=0$}\\\frac x{\norm{x}}&\mbox{otherwise.}\end{cases}\]
  By virtue of Theorem \ref{ThSubdiffConv}, $\partial q$ is an \SCD mapping. When considering a pair $(\xb,\xba)\in\gph\partial q$ with $\xb=0$ and $\norm{\xba}<1$, then it is easy to see that $\partial q$ is graphically smooth of dimension $n$ at $(\xb,\xba)$ and, by Definition \ref{DefSCDProperty},
  \[\Sp\partial q(\xb,\xba)=\Sp^*\partial q(\xb,\xba)=\big\{\{0\}\times\R^n\big\}.\]
  In this case we have the representation $\{0\}\times\R^n=\rge(B,I-B)$ with $B=0$.
  If $x\not=0$ then $q$ is even twice continuously differentiable near $x$ and, as pointed out below Theorem \ref{ThSubdiffConv}, with $\xba=\frac x{\norm{x}}$ one has
  \[\Sp\partial q(\xb,\xba)=\Sp^*\partial q(\xb,\xba)=\rge(B_x,I-B_x)=\rge(I,\nabla^2 q(x))\]
  with
  \begin{align*}
    B_x=\big(I+\nabla^2 q(x)\big)^{-1}=\left(I+\frac1{\norm{x}}\Big(I-\frac {xx^T}{\norm{x}^2} \Big)\right)^{-1}=\left(\frac{\norm{x}+1}{\norm{x}}\Big(I-\frac{xx^T}{\norm{x}^2(1+\norm{x})}\Big)\right)^{-1}.
  \end{align*}
  We claim that
  \begin{equation}\label{EqB_x}B_x=\frac{\norm{x}}{\norm{x}+1}\Big(I+\frac{xx^T}{\norm{x}^3}\Big).\end{equation}
  Indeed,
  \begin{align*}\Big(I-\frac{xx^T}{\norm{x}^2(1+\norm{x})}\Big)\Big(I+\frac{xx^T}{\norm{x}^3}\Big)&= I+xx^T\Big(\frac1{\norm{x}^3}-\frac1{\norm{x}^2(1+\norm{x})}-\frac{\norm{x}^2}{\norm{x^5}(1+\norm{x})}\Big)\\ &=I+xx^T\Big(\frac{1+\norm{x}-\norm{x}-1}{\norm{x}^3(1+\norm{x})}\Big)=I,\end{align*}
  and so formula \eqref{EqB_x} holds true.

  Finally consider the point $(\xb,\xba)$ with $\xb=0$ and $\norm{\xba}=1$.  By Definition \ref{DefSCDProperty} and Theorem \ref{ThSubdiffConv} one has that
  \[\Sp\partial q(\xb,\xba)=\Sp^*\partial q(\xb,\xba)=\big\{\{0\}\times\R^n\big\}\;\cup\;\Limsup_{\AT{x\to 0,x\not=0}{\frac x{\norm{x}}\to\xba}}\;\rge(B_x,I-B_x).\]
  Since the matrices $B_x$ are bounded, the above $\Limsup$ amounts to $\rge(B,I-B)$ where, taking into account \eqref{EqB_x},
  \[B=\lim_{\AT{x\to 0,x\not=0}{\frac x{\norm{x}}\to\xba}}B_x= \xba\xba{}^T\]
    However, note that
    \[\lim_{\AT{x\to 0,x\not=0}{\frac x{\norm{x}}\to\xba}}\nabla^2 q(x)\]
    does not exist.
  Finally note that, at points $(\xb,\xba)$ with $\xb=0$ and $\norm{\xba}=1$, one has
  \[\gph D^*\partial q(\xb,\xba)=\Sp^*\partial q(\xb,\xba)\cup\{(s,s^*)\mv s\in\R_-\{\xba\}, \skalp{s^*,\xba}\leq 0\},\]
  where the last term is generated by sequences $(0,x^*)\to(0,\xba)$ with $\norm{x^*}=1$. Thus, in this situation the mapping $\Sp^*\partial q(\xb,\xba)$ has a simpler structure than the limiting coderivative $D^*\partial q(\xb,\xba)$ (similarly as in \cite[Example 3.29]{GfrOut21a}.)
\end{example}
\if{
The representation of $L$ via the matrix $B$ is mainly of theoretical interest, for actual computations the use of another basis of $L$ is probably more appropriate. E.g., if $q$ is twice continuously differentiable then $\rge(I,\nabla^2 q(x))=\gph D^*\partial q(x,\nabla q(x))$ and the relation between $B$ and $\nabla^2q(x)$ is given by $B=(I+\nabla^2 q(x))^{-1}$, $I-B=(I+\nabla^2 q(x))^{-1}\nabla^2 q(x)$ and $\nabla^2 q(x)=B^{-1}(I-B)$. It seems to be advantageous to work with $I$ and $\nabla^2 q(x)$ instead of $B$ and $I-B$.\\
Therefore, let us now discuss another possible representation of the subspace  $L=\rge(B,I-B)$.
\begin{lemma}\label{LemBases}
\begin{enumerate}
\item Let $B$ be a symmetric positive semidefinite $n\times n$ matrix with $\norm{B}\leq 1$ and let $\V:=\ker B$ and $\U=\V^\perp=\rge B$. Further let $P$ denote the symmetric $n\times n$ matrix representing the orthogonal projection on $\U$ and set $W:=B^\dag (I-B) + (I-P)$, where $B^\dag$ denotes the Moore-Penrose inverse of $B$. Then
    \[\rge(B,I-B)=\rge(P,W),\]
  where $W$ is symmetric positive semidefinite and satisfies $W(I-P)=(I-P)W=I-P$.
\item Let $P$ be a symmetric $n\times n$ matrix representing the orthogonal projection on some subspace $\U$ and let $W$ be a symmetric positive semidefinite $n\times n$ matrix satisfying $W(I-P)=I-P$. Then $W+P$ is positive definite and the matrix $B:=(W+P)^{-1}P$ is symmetric and positive semidefinite, $\norm{B}\leq 1$ and
    \[\rge(B,I-B)=\rge(P,W).\]
\end{enumerate}
\end{lemma}
\begin{proof}
\begin{enumerate}
\item
  Taking into account $BB^\dag=B^\dag B=P$ and $B(I-P)=B^\dag(I-P)=0$, we obtain
\[\big(B+ I-P\big)\big(B^\dag+ I-P\big)=B B^\dag+(I-P)=I.\]
Hence $B+ I-P$ is nonsingular and
\begin{gather*}(B+I-P)^{-1}B=\big(B^\dag+(I-P)\big)B=P,\\
(B+I-P)^{-1}(I-B)=B^\dag (I-B) + (I-P)(I-B)=B^\dag (I-B) + (I-P)=W.
\end{gather*}
Since $B$ is symmetric, $B^\dag$ is symmetric as well and from the identity $B^\dag (I-B)=-B^\dag B+ B^\dag=-P+B^\dag=(I-B)B^\dag$ we conclude that $W$ is symmetric. Further, both matrices $I-B$ and $B^\dag$ are symmetric and positive semidefinite and since they commute, their product is also positive semidefinte. Since $I-P$ is also positive semidefinite, so must be $W$ as well. From $B^\dag(I-P)=0$ we obtain $W(I-P)=(I-P)=(I-P)^TW^T=(I-P)W$.
\item Since $W$ and $I-P$ are symmetric, $W(I-P)=(I-P)=(W(I-P))^T=(I-P)W$ implying that $W$ and $P$ commute. Since $(I-P)(W+P)P=(I-P)P+(I-P)P^2=0$ and $P(W+P)(I-P)=P(I-P)+P^2(I-P)=0$, for every $u\in\R^n$ we  obtain
\begin{align*}\skalp{u, (W+P)u}&=\skalp{Pu,(W+P)Pu}+\skalp{(I-P)u,(W+P)Pu}+\skalp{Pu,(W+P)(I-P)u}\\&\qquad\qquad+\skalp{(I-P)u,(W+P)(I-P)Pu}\\
&=\skalp{Pu,WPu}+\norm{Pu}^2+\norm{(I-P)u}^2\geq \norm{u}^2\end{align*}
showing that $W+P$ is positive definite. Since $W$ and $P$ commute, it follows that $W+P$  commutes both with $W$ and $P$ and consequently so does $(W+P)^{-1}$.
Hence, both $B=(W+P)^{-1}P$ and $I-B=I-(W+P)^{-1}P=(W+P)^{-1}W$ are  symmetric and positive semidefinite and  are symmetric and positive semidefinite and from the definiteness of $I-B$ we deduce $\norm{B}\leq 1$. By observing that $\rge(P,W)=\rge(P(W+P)^{-1},W(P+W)^{-1})=\rge((W+P)^{-1}P,(P+W)^{-1}W)=\rge(B,I-B)$, the proof is complete.
\end{enumerate}
\end{proof}
Let us illustrate the above development via a simple example.
\begin{example}
  Assume that $q(x)=\norm{x}$ so that
  \[\partial q(x)=\begin{cases}\B&\mbox{for $x=0$}\\\frac x{\norm{x}}&\mbox{otherwise.}\end{cases}\]
  By virtue of Theorem \ref{ThSubdiffConv}, $\partial q$ is an \SCD mapping. When considering a pair $(\xb,\xba)\in\gph\partial q$ with $\xb=0$ and $\norm{\xba}<1$, then it is easy to see that $\partial q$ is graphically smooth of dimension $n$ at $(\xb,\xba)$ and, by Definition \ref{DefSCDProperty},
  \[\Sp\partial q(\xb,\xba)=\Sp^*\partial q(\xb,\xba)=\big\{\{0\}\times\R^n\big\}.\]
  In this case we have the representation $\{0\}\times\R^n=\rge(B,I-B)=\rge(P,W)$ with $B=P=0$. $\V=\R^n$, $\U=\{0\}$ and $W=I$.
  If $x\not=0$ then $q$ is even twice continuously differentiable near $x$ and, as pointed out below Theorem \ref{ThSubdiffConv}, with $\xba=\frac x{\norm{x}}$ one has
  \[\Sp\partial q(\xb,\xba)=\Sp^*\partial q(\xb,\xba)=\rge(B_x,I-B_x)=\rge(P_x,W_x)\]
  with $P_x=I$, $\U=\R^n$, $\V=0$, $W_x=\nabla^2q(x)=\frac1{\norm{x}}\Big(I-\frac {xx^T}{\norm{x}^2}\Big)$ and
  \begin{align*}
    B_x=\big(I+\nabla^2 q(x)\big)^{-1}=\left(I+\frac1{\norm{x}}\Big(I-\frac {xx^T}{\norm{x}^2} \Big)\right)^{-1}=\left(\frac{\norm{x}+1}{\norm{x}}\Big(I-\frac{xx^T}{\norm{x}^2(1+\norm{x})}\Big)\right)^{-1}.
  \end{align*}
  We claim that
  \begin{equation}\label{EqB_x}B_x=\frac{\norm{x}}{\norm{x}+1}\Big(I+\frac{xx^T}{\norm{x}^3}\Big).\end{equation}
  Indeed,
  \begin{align*}\Big(I-\frac{xx^T}{\norm{x}^2(1+\norm{x})}\Big)\Big(I+\frac{xx^T}{\norm{x}^3}\Big)&= I+xx^T\Big(\frac1{\norm{x}^3}-\frac1{\norm{x}^2(1+\norm{x})}-\frac{\norm{x}^2}{\norm{x^5}(1+\norm{x})}\Big)\\ &=I+xx^T\Big(\frac{1+\norm{x}-\norm{x}-1}{\norm{x}^3(1+\norm{x})}\Big)=I,\end{align*}
  and so formula \eqref{EqB_x} holds true.

  Finally consider the point $(\xb,\xba)$ with $\xb=0$ and $\norm{\xba}=1$.  By Definition \ref{DefSCDProperty} and Theorem \ref{ThSubdiffConv} one has that
  \[\Sp\partial q(\xb,\xba)=\Sp^*\partial q(\xb,\xba)=\big\{\{0\}\times\R^n\big\}\;\cup\;\Limsup_{\AT{x\to 0,x\not=0}{\frac x{\norm{x}}\to\xba}}\;\rge(B_x,I-B_x).\]
  Since the matrices $B_x$ are bounded, the above $\Limsup$ amounts to $\rge(B,I-B)$ where, taking into account \eqref{EqB_x},
  \[B=\lim_{\AT{x\to 0,x\not=0}{\frac x{\norm{x}}\to\xba}}B_x= \xba\xba{}^T\]
  In terms of Lemma \ref{LemBases} we have $\rge(B,I-B)=rge(P,W)$ with $P=B$ and $W=I-P$, but there does not hold
  \[P=\lim_{\AT{x\to 0,x\not=0}{\frac x{\norm{x}}\to\xba}}P_x,\quad W=\lim_{\AT{x\to 0,x\not=0}{\frac x{\norm{x}}\to\xba}}W_x.\]
  Finally note that, at points $(\xb,\xba)$ with $\xb=0$ and $\norm{\xba}=1$, one has
  \[\gph D^*\partial q(\xb,\xba)=\Sp^*\partial q(\xb,\xba)\cup\{(s,s^*)\mv s\in\R_-\xba, \skalp{s^*,\xba}\leq 0\},\]
  where the last term is generated by sequences $(0,x^*)\to(0,\xba)$ with $\norm{x^*}=1$. Thus, in this situation the mapping $\Sp^*\partial q(\xb,\xba)$ has a simpler structure than the limiting coderivative $D^*\partial q(\xb,\xba)$ (similarly as in \cite[Example 3.29]{GfrOut21a}.)
\end{example}
Note that the $\V\U$-space decomposition provided by Lemma \ref{LemBases} depends on $\ker B$ and not on $\Span \partial q(x)$ as the one introduced in \cite{LeOuSz00}. Therefore these $\V\U$-space decompositions might be different.

The alternative representation of subspaces $L\in\Sp\partial q(x,x^*)$ given by Lemma \ref{LemBases} can be geometrically interpreted as follows.

\begin{proposition}Let $q:\R^n\to\bar \R$ be a proper lsc convex function and let $L\in\Sp \partial q(\xb,\xba)$. Let $L=\rge(P,W)$ , where $P$ and $W$ are symmetric positive semidefinite $n\times n$ matrices such that $P$ represents the orthogonal projection on some subspace $\U$ and $W(I-P)=I-P$. Let $G$ be a symmetric $n\times n$ matrix with $P(G-W)P=0$ and consider the function $\tilde q:\R^n\to\bar \R$ given by
\[\tilde q(x)=q(\xb)+\skalp{\xba,x-\xb}+\frac 12\skalp{x-\xb,G(x-\xb)}+\delta_\U(x-\xb).\]
Then $\tilde q$ is convex, $\tilde q(\xb)=q(\xb)$, $\xba\in \partial \tilde q(\xb)$ and for  $(x,x^*)\in\gph\partial \tilde q$ one has $\gph D\partial\tilde q(x,x^*)=\rge(P,W)$, $\Sp \partial\tilde q(x,x^*)=\{\rge(P,W)\}$.
\end{proposition}
\begin{proof}Since $\dom \tilde q=\xb+\U$ and $u^TGu=u^TWu\geq 0$, $u\in\U$, we have
$\tilde q(x)=q(\xb)+\skalp{\xba,x-\xb}+\frac 12\skalp{x-\xb,W(x-\xb)}+\delta_\U(x-\xb)$, $x\in\R^n$. Hence $\tilde q$ is convex as the sum of two convex functions and $\tilde q(\xb)=q(\xb)$ and $\xba\in\partial \tilde q(\xb)$ follows.
  For every $x\in\xb+\U$ we have $\partial \delta_\U(x-\xb)=N_\U(x-\xb)=\U^\perp$ implying that $\partial \tilde q(x)=\xba+W(x-\xb)+\U^\perp$. Hence, for every $x^*\in\partial \tilde q(x)$ we obtain
  \[T_{\gph \partial\tilde q}(x,x^*)=\{(u,Wu)\mv u\in\U\}+\{0\}\times \U^\perp=:\bar L\]
  and   $\Sp \partial\tilde q(x,x^*)=\{\bar L\}$, $(x,x^*)\in\gph \partial \tilde q$ follows. Since $W(I-P)=I-P$ and $\U^\perp=\rge(I-P)$, we obtain
  \[\bar L=\{(Pp,WPp)\mv p\in\R^n\}+\{(0)\}\times\rge(I-P)=\rge(P,WP)+\rge(0,W(I-P))=\rge(P,W)\]
  and the proof is complete.
\end{proof}
\begin{remark}
  In \cite{Ro85}, a function of the form
  \[\frac 12u^TGu+\delta_\U(u)\]
  is called a {\em generalized purely quadratic function}.
\end{remark}
}\fi
In our numerical experiments we will use convex functions with some separable structure, which carries over to $\Sp^*\partial q$.
\begin{lemma}\label{LemSepStruct}
If $q(x_1,x_2)=q(x_1)+q_2(x_2)$ for lsc convex functions $q_i:\R^{n_i}\to\bar\R$, $i=1,2$, then for every $((\xb_1,\xb_2),(\xb_1^*,\xb_2^*))\in\gph \partial q$ there holds
\begin{equation*}
\Sp \partial q((\xb_1,\xb_2),(\xb_1^*,\xb_2^*))=\Big\{\{((u_1,u_2), (u_1^*,u_2^*))\mv (u_i,u_i^*)\in L_i, i=1,2\}\mv L_i\in \Sp\partial q_i(\xb_i,\xb_i^*),\ i=1,2\Big\}.
\end{equation*}
\end{lemma}
\begin{proof}
  We claim that $\OO_{\partial q}=\{((x_1,x_2),(x_1^*,x_2^*))\mv (x_i,x_i^*)\in\OO_{\partial q_i},\ i=1,2\}$ and that
  \begin{equation}\label{EqAuxTanCone}T_{\gph \partial q}((x_1,x_2),(x_1^*,x_2^*))=\{((u_1,u_2), (u_1^*,u_2^*))\mv (u_i,u_i^*)\in T_{\gph \partial q_i}(x_i,x_i^*),\ i=1,2\}\end{equation}
  holds for all $((x_1,x_2),(x_1^*,x_2^*))\in\OO_{\partial q}$. Indeed, if $((x_1,x_2),(x_1^*,x_2^*))\in\OO_{\partial q}$ then $(x_1,x_1^*)\in \OO_{\partial q_1}$ because of $\{((u_1,0),(u_1^*,0))\mv (u_1,u_1^*)\in\partial q_1(x_1,x_1^*)\}\subseteq T_{\gph \partial q}((x_1,x_2),(x_1^*,x_2^*))$ and, analogously, $(x_2,x_2^*)\in \OO_{\partial q_1}$. This proves $\OO_{\partial q}\subseteq\{((x_1,x_2),(x_1^*,x_2^*))\mv (x_i,x_i^*)\in\OO_{\partial q_i},\ i=1,2\}$. To show the reverse inclusion, consider $(x_i,x_i^*)\in \OO_{q_i}$, $i=1,2$. Taking into account \cite[Corollary 3.28, Remark 3.18]{GfrOut21a}, the sets $\gph\partial q_i$ are geometrically derivable at points $(x_i,x_i^*)\in \OO_{\partial q_i}$, $i=1,2$ and therefore
  \begin{equation}\label{EqAuxTanCone1}
  T_{\gph \partial q_1\times \gph \partial q_2}((x_1,x_1^*),(x_2,x_2^*))=T_{\gph \partial q_1}(x_1,x_1^*)\times T_{\gph \partial q_2}(x_2,x_2^*)
  \end{equation}
  by \cite[Proposition 1]{GfrYe17a}. Thus, $T_{\gph \partial q_1\times \gph \partial q_2}((x_1,x_1^*),(x_2,x_2^*))$ is an $n_1+n_2$ dimensional subspace and, since the tangent cones in \eqref{EqAuxTanCone} and \eqref{EqAuxTanCone1} coincide up to a reordering of the elements, $((x_1,x_2),(x_1^*,x_2^*))\in \OO_{\partial q}$ together with the validity of \eqref{EqAuxTanCone} follows. Hence our claim holds true and the assertion of the lemma follows from the definition.
\end{proof}
Clearly, the assertion of Lemma 2.10 can be extended to the general case when the sum defining $q$ has an arbitrary finite number of terms.

\section{On \ssstar Newton methods for \SCD mappings}

In this section we recall the general framework for the \ssstar Newton method  introduced in \cite{GfrOut21} and adapted to \SCD mappings in \cite{GfrOut21a}.
Consider the inclusion
\begin{equation}\label{EqIncl}
  0\in F(x),
\end{equation}
where $F:\R^n\tto\R^n$ is a mapping having the \SCD property around some point $(\xb,0)\in\gph F$.
\begin{definition}
We say that $F:\R^n\tto\R^n$ is {\em\SCD \ssstar}  at $(\xb,\yb)\in\gph F$ if $F$ has the \SCD property around $(\xb,\yb)$ and
for every $\epsilon>0$ there is some $\delta>0$ such that the inequality
\begin{align*}
\vert \skalp{x^*,x-\xb}-\skalp{y^*,y-\yb}\vert&\leq \epsilon
\norm{(x,y)-(\xb,\yb)}\norm{(x^*,y^*)}
\end{align*}
holds for all $(x,y)\in \gph F\cap \B_\delta(\xb,\yb)$ and all $(y^*,x^*)$ belonging to any $L\in\Sp^*F(x,y)$.
\end{definition}
 Clearly, every mapping with the SCD property around  $(\xb,\yb) \in\gph F$
which is \ssstar at  $(\xb,\yb)$ is automatically SCD \ssstar at  $(\xb,\yb)$. Therefore, the class of \SCD \ssstar mappings is even richer than the class of \ssstar maps. In particular, it follows from \cite[Theorem 2]{Jou07} that every mapping whose graph is a closed subanalytic set is \SCD \ssstar, cf. \cite{GfrOut21a}.

The following proposition provides the key estimate for the \ssstar Newton method for \SCD mappings.
\begin{proposition}[cf. {\cite[Proposition 5.3]{GfrOut21a}}]\label{PropConvNewton}
  Assume that $F:\R^n\tto\R^n$ is \SCD \ssstar at $(\xb,\yb)\in\gph F$. Then for every  $\epsilon>0$ there is  some $\delta>0$ such that the estimate
  \begin{equation*}
  \norm{x-C_L^T(y-\yb)-\xb}\leq \epsilon\sqrt{n(1+\norm{C_L}^2)}\norm{(x,y)-(\xb,\yb)}
  \end{equation*}
  holds for every $(x,y)\in\gph F\cap \B_\delta(\xb,\yb)$ and every $L\in\Sp^*F(x,y)\cap\Z_n^{\rm reg}$.
\end{proposition}
We now describe the \SCD variant of the \ssstar Newton method. Given a solution $\xb\in F^{-1}(0)$ of \eqref{EqIncl} and some positive scalar, we define the mappings $\A_{\eta,\xb}:\R^n\tto\R^n\times\R^n$ and $\New_{\eta,\xb}:\R^n\tto\R^n$ by
\begin{gather*}
  \A_{\eta,\xb}(x):=\{(\hat x,\hat y)\in\gph F\mv \norm{(\hat x,\hat y)-(\xb,0)}\leq \eta\norm{x-\xb}\},\\
  \New_{\eta,\xb}(x):=\{\hat x-C_L^T\hat y\mv (\hat x,\hat y)\in \A_{\eta,\xb}(x), L\in\Sp^*F(\hat x,\hat y)\cap \Z_n^{\rm reg}\}.
\end{gather*}
\begin{proposition}\label{PropSingleStep}
  Assume that $F$ is \SCD \ssstar at $(\xb,0) \in\gph F$ and \SCD regular around $(\xb,0)$ and let $\eta>0$. Then there is some $\bar\delta>0$  such that for every $x\in \B_{\bar\delta}(\xb)$ the mapping $F$ is \SCD regular around every point $(\hat x,\hat y)\in \A_{\eta,\xb}(x)$. Moreover, for every $\epsilon>0$ there is some $\delta\in(0,\bar\delta]$ such that
  \[\norm{z-\xb}\leq\epsilon\norm{x-\xb}\ \forall x\in \B_\delta(\xb), \forall z\in \New_{\eta,\xb}(x).\]
\end{proposition}
\begin{proof}
  Let $\kappa:={\rm scd\,reg\;}F(\xb,0)$. Then, by Lemma \ref{LemSCDReg} there is some $\delta'>0$ such that $F$ is \SCD regular with ${\rm scd\,reg\;}F(x,y)\leq \kappa+1$ around any $(\hat x,\hat y)\in\gph F\cap \B_{\delta'}(\xb,0)$ and the first assertion follows with $\bar\delta:=\delta'/\eta$. Now consider $\epsilon>0$ and set $\tilde \epsilon:=\epsilon/(\eta\sqrt{n(1+(1+\kappa)^2)})$. By Proposition \ref{PropConvNewton} there is some $\tilde \delta\in (0,\delta']$ such that the inequality
  \[\norm{\hat x-C_L^T\hat y}\leq \tilde \epsilon\sqrt{n(1+\norm{C_L}^2)}\norm{(\hat x,\hat y)-(\xb,0)}\]
  holds for every $(\hat x,\hat y)\in\gph F\cap \B_{\tilde \delta}$ and every $L\in \Sp^*F(\hat x,\hat y)\cap\Z_n^{\rm reg}$. Set $\delta:=\tilde\delta/\eta$ and consider $x\in \B_\delta(\xb)$. For every $(\hat x,\hat y)\in \A_{\eta,\xb}(x)$ we have $\norm{(\hat x,\hat y)-(\xb,0)}\leq \eta\norm{x-\xb}\leq \tilde \delta\leq\delta'$ and consequently
  \[\norm{C_L}\leq {\rm scd\,reg\;}F(\hat x,\hat y)\leq \kappa+1\ \forall L\in \Sp^*F(\hat x,\hat y)\]
  Thus
  \[\norm{\hat x-C_L^T\hat y}\leq \tilde \epsilon\sqrt{n(1+(1+\kappa)^2)}\norm{(\hat x,\hat y)-(\xb,0)}\leq \epsilon\norm{x-\xb}\]
  and the second assertion follows.
\end{proof}
Assuming we are given some iterate $x^{(k)}$, the next iterate is formally given by $\ee x{k+1}\in \New_{\eta,\xb}(\ee xk)$. Let us have a closer look at this rule. Since we cannot expect in general that $F(x^{(k)})\not=\emptyset$ or that $0$ is close to
$F(x^{(k)})$, even if $x^{(k)}$ is close to a solution $\xb$,  we first perform  some  step which
yields $(\hat x^{(k)},\hat y^{(k)})\in\gph F$ as  an approximate projection of $(x^{(k)},0)$ onto $\gph F$. We require that
\begin{equation}\label{EqBndApprStep}
\norm{(\hat x^{(k)},\hat y^{(k)})-(\xb,0)}\leq \eta\norm{x^{(k)}-\xb}
\end{equation}
for some constant $\eta>0$, i.e. $(\ee{\hat x}k,\ee{\hat y}k)\in\A_{\eta,\xb}(\ee xk)$. For instance, if
\[\norm{(\hat x^{(k)},\hat y^{(k)})-(x^{(k)},0)}\leq \beta\dist{(x^{(k)},0),\gph F}\]
holds with some $\beta\geq 1$, then
\begin{align*}\norm{(\hat x^{(k)},\hat y^{(k)})-(\xb,0)}&\leq \norm{(\hat x^{(k)},\hat y^{(k)})-(x^{(k)},0)}+\norm{(x^{(k)},0)-(\xb,0)}\\
&\leq  \beta\dist{(x^{(k)},0),\gph F}+\norm{(x^{(k)},0)-(\xb,0)}\leq (\beta+1)\norm{(x^{(k)},0)-(\xb,0)}
\end{align*}
and thus \eqref{EqBndApprStep} holds with $\eta=\beta+1$ and we can fulfill \eqref{EqBndApprStep} without knowing the solution $\xb$. Further we require that $\Sp^*F(\hat x^{(k)},\hat y^{(k)})\cap\Z_n^{\rm reg}\not=\emptyset$ and  compute the new iterate as $x^{(k+1)}=\hat x^{(k)}-C_L^T\hat y^{(k)}$ for some $L\in \Sp^*F(\hat x^{(k)},\hat y^{(k)})\cap\Z_n^{\rm reg}$. In fact, in our numerical implementation we will not compute the matrix $C_L$, but two $n\times n$ matrices $A,B$ such that $L=\rge(B^T,A^T)$. The next iterate $x^{(k+1)}$ is then obtained by  $x^{(k+1)}=\hat x^{(k)}+\Delta x^{(k)}$ where  $\Delta x^{(k)}$ is a solution of the system $A\Delta x=-B\hat y^{(k)}$.
\if{
An  iteration step of the \SCD variant of the \ssstar Newton method can now be described as follows. Assume we are given some iterate $x^{(k)}$. Since we cannot expect in general that $F(x^{(k)})\not=\emptyset$ or that $0$ is close to
$F(x^{(k)})$, even if $x^{(k)}$ is close to a solution $\xb$,  we first perform  some  step which
yields $(\hat x^{(k)},\hat y^{(k)})\in\gph F$ as  an approximate projection of $(x^{(k)},0)$ onto $\gph F$. We require that
\begin{equation}\label{EqBndApprStep}
\norm{(\hat x^{(k)},\hat y^{(k)})-(\xb,0)}\leq \eta\norm{x^{(k)}-\xb}
\end{equation}
for some constant $\eta>0$. E.g., if
\[\norm{(\hat x^{(k)},\hat y^{(k)})-(x^{(k)},0)}\leq \beta\dist{(x^{(k)},0),\gph F}\]
holds with some $\beta\geq 1$, then
\begin{align*}\norm{(\hat x^{(k)},\hat y^{(k)})-(\xb,0)}&\leq \norm{(\hat x^{(k)},\hat y^{(k)})-(x^{(k)},0)}+\norm{(x^{(k)},0)-(\xb,0)}\\
&\leq  \beta\dist{(x^{(k)},0),\gph F}+\norm{(x^{(k)},0)-(\xb,0)}\leq (\beta+1)\norm{(x^{(k)},0)-(\xb,0)}
\end{align*}
and thus \eqref{EqBndApprStep} holds with $\eta=\beta+1$. Further we require that $\Sp^*F(\hat x^{(k)},\hat y^{(k)})\cap\Z_n^{\rm reg}\not=\emptyset$ and  compute the new iterate as $x^{(k+1)}=\hat x^{(k)}-C_L^T\hat y^{(k)}$ for some $L\in \Sp^*F(\hat x^{(k)},\hat y^{(k)})\cap\Z_n^{\rm reg}$. In fact, in our numerical implementation we will not compute the matrix $C_L$, but two $n\times n$ matrices $A,B$ such that $L=\rge(B^T,A^T)$. The next iterate $x^{(k+1)}$ is then obtained by  $x^{(k+1)}=\hat x^{(k)}+\Delta x^{(k)}$ where  $\Delta x^{(k)}$ is a solution of the system $A\Delta x=-B\hat y^{(k)}$.
}\fi
This leads to the following conceptual algorithm.
\begin{algorithm}[\SCD \ssstar Newton-type method for inclusions]\label{AlgNewton}\mbox{ }\\
 1. Choose a starting point $x^{(0)}$, set the iteration counter $k:=0$.\\
 2. If ~ $0\in F(x^{(k)})$, stop the algorithm.\\
  3. \begin{minipage}[t]{\myAlgBox} {\bf Approximation step: } Compute
  $$(\hat x^{(k)},\hat y^{(k)})\in\gph F$$ satisfying \eqref{EqBndApprStep} and such that $\Sp^*F(\hat x^{(k)},\hat y^{(k)})\cap\Z_n^{\rm reg}\not=\emptyset$.\end{minipage}\\
  4. \begin{minipage}[t]{\myAlgBox} {\bf Newton step: }Select $n\times n$ matrices $A^{(k)},B^{(k)}$ with
  $$L^{(k)}:=\rge\big({B^{(k)}}^T,{A^{(k)}}^T)\in \Sp^*F(\hat x^{(k)},\hat y^{(k)})\cap\Z_n^{\rm reg},$$ calculate the Newton direction $\Delta x^{(k)}$ as a solution of the linear system $$A^{(k)}\Delta x=-B^{(k)}\hat y^{(k)}$$ and obtain the new iterate via $x^{(k+1)}=\hat x^{(k)}+\Delta x^{(k)}.$\end{minipage}\\
  \strut5. Set $k:=k+1$ and go to 2.
\end{algorithm}
For this algorithm, locally superlinear convergence follows from Proposition \ref{PropSingleStep}, see also \cite[Corollary 5.6]{GfrOut21a}.
\begin{theorem}\label{ThConvSSNewton}
 Assume that $F$ is \SCD \ssstar at $(\xb,0) \in\gph F$ and \SCD regular around $(\xb,0)$. Then for every $\eta>0$ there is a neighborhood $U$ of $\xb$ such that
 for every starting point $x^{(0)}\in U$ Algorithm \ref{AlgNewton} is well-defined and either stops after finitely many iterations at a solution of \eqref{EqIncl} or produces a sequence $x^{(k)}$ converging superlinearly to $\xb$  for any choice of $(\hat x^{(k)} ,\hat y^{(k)})$ satisfying \eqref{EqBndApprStep} and any $L^{(k)}\in\Sp^*F(\hat x^{(k)} ,\hat y^{(k)})$.
\end{theorem}

 As shown in \cite[Corollary 6.4]{GfrOut21a}, if $F$ happens to be SCD \ssstar around $(\xb,0)$, then the assumptions of the above statement are fulfilled whenever $F$ is strongly metrically subregular at all points from a neighborhood of $(\xb,0)$. Hence, in particular, these assumptions are satisfied provided $F$ is strongly metrically regular around $(\xb,0)$, which is used in the test problem discussed in Section 7.

There is an alternative for the computation of the Newton direction $\Delta x^{(k)}$ based on the subspaces from $\Sp F(\ee{\hat x}k,\ee{\hat y}k)$, cf. \cite{GfrOut21a}:\\
4. \begin{minipage}[t]{\myAlgBox} {\bf Newton step: }\em Select $n\times n$ matrices $A^{(k)},B^{(k)}$ with
$$L^{(k)}:= \rge\big({A^{(k)}},{B^{(k)}})\in \Sp F(\hat x^{(k)},\hat y^{(k)})\cap\Z_n^{\rm reg},$$ compute a solution $p$ of the linear system
$${B^{(k)}}p =-\hat y^{(k)}$$  and obtain the new iterate $x^{(k+1)}=\hat x^{(k)}+\Delta x^{(k)}$ with Newton direction $\Delta x^{(k)}=A^{(k)}p$.\end{minipage} \\ \\

  For the choice between the two approaches for calculating the Newton direction it is important to consider whether elements from $\Sp^*F(\hat x^{(k)} ,\hat y^{(k)})$ or from $\Sp F(\hat x^{(k)} ,\hat y^{(k)})$ are easier to compute.

Note that for an implementation of the Newton step we need not to know the whole derivative $\Sp^*F(\hat x^{(k)},\hat y^{(k)})$ (or $\Sp F(\hat x^{(k)},\hat y^{(k)})$) but only one element $L^{(k)}\in \Sp^*F(\hat x^{(k)},\hat y^{(k)})$.

\section{Implementation of the \ssstar Newton method
}

There is a lot of possibilities how to implement the \ssstar Newton method. Apart from the Newton step, which is not uniquely determined by different choices of subspaces contained in $\Sp^* F(\hat x^{(k)},\hat y^{(k)})$, there is a multitude of possibilities how to perform the approximation step. In this section we will construct  an implementable version of the \ssstar Newton method for the numerical solution of GE \eqref{EqVI2ndK} under the assumption that the {\em proximal} mapping $P_\lambda q$, defined by
\[P_\lambda q(y):=\argmin_x\{\frac 1{2\lambda}\norm{x-y}^2+q(x)\},\ y\in\R^n,\]
can be efficiently evaluated for every $y\in\R^n$ and parameter $\lambda>0$. Since $q$ is convex, it is well known that for every $\lambda>0$ the proximal mapping $P_\lambda q$ is single-valued and nonexpansive and $P_\lambda q=(I+\lambda\partial q)^{-1}$, see, e.g. \cite[Proposition 12.19]{RoWe98}.

Given some scaling parameter $\gamma>0$, we will denote
\[u_\gamma(x):=P_{\frac 1\gamma}q(x-\frac 1\gamma f(x))-x.\]
From the definition of the proximal mapping we obtain that $u_\gamma(x)$ is the unique solution of the uniformly convex optimization problem
\[\min_u \frac \gamma2\norm{u}^2+\skalp{f(x),u}+q(x+u).\]
The first-order (necessary and sufficient) optimality condition reads as
\begin{equation}\label{EqOptCond_u}0\in \gamma u_\gamma(x)+f(x)+\partial q(x+u_\gamma(x)).\end{equation}
Since $P_\lambda q$ is nonexpansive, we obtain the bounds
\begin{align}
\nonumber  &\norm{(x+u_\gamma(x))-(x'+u_\gamma(x')}\leq \norm{(x-x')-\frac 1\gamma(f(x)-f(x'))}\leq \norm{x-x'}+\frac 1\gamma\norm{f(x)-f(x')},\\
\label{EqLip2}  &\norm{u_\gamma(x)-u_\gamma(x')}\leq  2\norm{x-x'}+\frac 1\gamma\norm{f(x)-f(x')}.
\end{align}

Our approach is based on an equivalent reformulation of \eqref{EqVI2ndK} in form of the GE
\begin{equation}\label{EqVI-alt}0\in\F(x,d):=\myvec{f(x)+\partial q(d)\\x-d}\end{equation}
in variables $(x,d)\in\R^n\times \R^n$. Clearly, $\xb$ is a solution of \eqref{EqVI2ndK} if and only if $(\xb,\xb)$ is a solution of \eqref{EqVI-alt}.
\begin{proposition}\label{PropF}\begin{enumerate}
\item[(i)]
  Let $x\in\R^n$, $(d,d^*)\in\gph\partial q$. Then
  \begin{align}
     \nonumber&\lefteqn{\Sp^* \F((x,d),(f(x)+d^*,x-d))}\\
    \label{EqSCD_F2}&=\left\{\rge\left(\left(\begin{matrix}Y^*&0\\0&-I\end{matrix}\right), \left(\begin{matrix}\nabla f(x)^TY^*&-I\\X^*&I\end{matrix}\right)\right)\mv \rge(Y^*,X^*)\in \Sp^*\partial q(d,d^*)\right\}.
  \end{align}
\item[(ii)]Let $\xb$ be a solution to \eqref{EqVI2ndK}. Then the following statements are equivalent:
    \begin{enumerate}
    \item[(a)] $H$ is \SCD regular around $(\xb,0)$.
    \item[(b)] For every $L\in \Sp^* \partial q(\xb, -f(\xb))$  and every  $(Y^*,X^*)\in\M(L)$ the matrix $\nabla f(x)^TY^*+X^*$ is nonsingular.
    \item[(c)] The mapping $\F$ is \SCD regular around $((\xb,\xb),(0,0))$.
    \end{enumerate}
\item[(iii)]Let $\xb$ be a solution to \eqref{EqVI2ndK}. If $\partial q$ is \SCD \ssstar at $(\xb,-f(\xb))$ then $\F$ is \SCD \ssstar at $((\xb,\xb),(0,0))$.
\end{enumerate}
\end{proposition}
\begin{proof} (i)  GE \eqref{EqVI-alt} can be written down in the form
\[0\in \F(x,d)=h(x,d)+F(x,d)\]
where $h(x,d):=(f(x),x-d)$ and $F(x,d):=\partial q(d)\times \partial g(x)$ with $g:\R^n\to\R$ given by $g(x)=0$ for all $x$. By virtue of \cite[Proposition 3.15]{GfrOut21a} we obtain that, at the point $\big((x,d),(f(x)+d^*,x-d)\big)\in\gph\F\subseteq \R^{2n}\times\R^{2n}$ one has
\[\Sp^*\F\big((x,d),(f(x)+d^*,x-d)\big)=\left(\begin{matrix}
  I&0\\
  \nabla h(x,d)^T&I
\end{matrix}\right)\Sp^*F\big((x,d),(d^*,0)\big).\]
Next consider the mapping $G:\R^{2n}\tto\R^{2n}$ given by $G(x,d)=\partial\big(g(x)+q(d)\big)$. Since
\begin{align*}&\gph F=\{\big((x,d),(d^*,x^*))\mv (d,d^*)\in\gph\partial q, (x,x^*)\in\gph\partial g\},\\
 &\gph G=\{\big((x,d),(x^*,d^*))\mv (d,d^*)\in\gph\partial q, (x,x^*)\in\gph\partial g\},\end{align*}
we can employ \cite[Proposition 3.14]{GfrOut21a} with $\Phi(x,d,d^*,x^*):=(x,d,x^*,d^*)$ to obtain that
\[\Sp^*F\big((x,d),(d^*,0))=S_{2n}\nabla\Phi(x,d,d^*,0)^TS_{2n}^T\Sp^*G(\big((x,d),(0,d^*)\big).\]
It remains to compute $\Sp^*G(\big((x,d),(0,d^*)\big)$. By virtue of Theorem \ref{ThSubdiffConv} and Lemma \ref{LemSepStruct} we have
\[\Sp^*G\big((x,d),(0,d^*)\big)=\left\{\rge\left(\left(\begin{matrix}I&0\\0&Y^*\end{matrix}\right),\left(\begin{matrix}0&0\\0&X^*\end{matrix}\right)\right)\mv \rge(Y^*,X^*)\in\Sp^*\partial q(d,d^*)\right\}.\]
Putting these ingredients together we may conclude that
\begin{align*}\Sp^*\F\big((x,d),&(f(x)+d^*,x-d)\big)=\left(\begin{matrix}
  I&0\\
  \nabla h(x,d)^T&I
\end{matrix}\right)S_{2n}\nabla\Phi(x,d,d^*,0)^TS_{2n}^T\Sp^*G(\big((x,d),(0,d^*)\big)\\
&=\left(\begin{matrix}
  0&I&0&0\\
  I&0&0&0\\
  I&\nabla f(x)^T&I&0\\
  -I&0&0&I
\end{matrix}\right)\Sp^*G\big((x,d),(0,d^*)\big)\\
&=\left\{\rge\left[\left(\begin{matrix}
  0&I&0&0\\
  I&0&0&0\\
  I&\nabla f(x)^T&I&0\\
  -I&0&0&I\end{matrix}\right)\left(\begin{matrix}
    I&0\\0&Y^*\\0&0\\0&X^*
  \end{matrix}\right)\right]\mv \rge(Y^*,X^*)\in \Sp^*\partial q(d,d^*)\right\}\\
  &=\left\{\rge\left(\left(\begin{matrix}
    0&Y^*\\I&0
  \end{matrix}\right),\left(\begin{matrix}
    I&\nabla f(x)^TY^*\\-I&X^*\end{matrix}\right)\right)\mv\rge(Y^*,X^*)\in \Sp^*\partial q(d,d^*)\right\}\\
  &=\left\{\rge\left(\left(\begin{matrix}
    0&Y^*\\I&0
  \end{matrix}\right)S_n,\left(\begin{matrix}
    I&\nabla f(x)^TY^*\\-I&X^*\end{matrix}\right)S_n\right)\mv\rge(Y^*,X^*)\in \Sp^*\partial q(d,d^*)\right\}
\end{align*}
leading to formula \eqref{EqSCD_F2}.

(ii) By \cite[Proposition 3.15]{GfrOut21a} we have
\begin{align*}\Sp^* H(\xb,0)&=\left(\begin{matrix}
  I&0\\\nabla f(\xb)^T&I\end{matrix}\right)\Sp^*\partial q(\xb,-f(\xb))\\
  &=\big\{\rge(Y^*,\nabla f(\xb)^TY^*+X^*)\mv \rge (Y^*,X^*)\in \Sp^*\partial q(\xb,-f(\xb))\big\}
\end{align*}
and the equivalence between (a) and (b) is implied by Proposition \ref{PropC_L}. By \eqref{EqSCD_F2}, the mapping $\F$ is \SCD regular around $((\xb,\xb),(0,0))$ if and only if for every pair $Y^*,X^*$ with $\rge(Y^*,X^*)\in \Sp^*\partial q(\xb,-f(\xb))$ the matrix
\[\left(\begin{matrix}\nabla f(x)^TY^*&-I\\X^*&I\end{matrix}\right)=\left(\begin{matrix}\nabla f(x)^TY^*+X^*&-I\\0&I\end{matrix}\right)\left(\begin{matrix}I&0\\X^*&I\end{matrix}\right)\]
is nonsingular and, by the representation above, this holds if and only if $\nabla f(x)^TY^*+X^*$ is nonsingular. Hence, (b) is equivalent to (c).

(iii) Let $f$ be Lipschitz continuous with constant $l$ in some  ball $\B_r(\xb)$ around $\xb$. Consider $\epsilon>0$, choose $\delta_q>0$ such that
\begin{align*}
  &\vert \skalp{e^*, d-\xb}-\skalp{e,d^*+f(\xb)}\vert\leq \frac\epsilon{2\sqrt{2}(l+1)} \norm{(e,e^*)}\norm{(d-\xb,d^*+f(\xb))}\\
   &\qquad\qquad\mbox{for all }(d,d^*)\in\gph \partial q\cap \B_{\delta_q}(\xb,-f(\xb))\mbox{ and all } (e,e^*)\in L\in\Sp^*\partial q(d,d^*)
\end{align*}
and then choose $\delta\leq \min\{\frac {\delta_q}{1+l},r\}$ such that
\[\norm{f(x)-f(\xb)-\nabla f(x)(x-\xb)}\leq \frac\epsilon{2\sqrt{2}(l+1)}\norm{x-\xb},\ x\in \B_\delta(\xb).\]
Consider $((x,d),(y_1,y_2))\in\gph \F\cap \B_\delta((\xb,\xb),(0,0))$, $((z_1,z_2),(z_1^*,z_2^*))\in \bar L\in \Sp^*\F((x,d),(y_1,y_2))$.
Then $y_1=f(x)+d^*$ with $d^*\in\partial q(d)$, $y_2=x-d$ and by \eqref{EqSCD_F2} there are $(e,e^*)\in L\in \Sp^*\partial q(d,d^*)$, $c\in\R^n$ with
$((z_1,z_2),(z_1^*,z_2^*))=((e,-c), (\nabla f(x)^Te-c,e^*+c))$. Then $\norm{x-\xb}\leq \delta$ and
\begin{align*}\norm{(d-\xb,d^*+f(\xb))}&\leq\norm{(d-\xb,y_1)}+\norm{f(x)-f(\xb)}\\
&\leq \norm{((x,d),(y_1,y_2))-((\xb,\xb),(0,0))}+\norm{f(x)-f(\xb)}\leq \delta+l\delta\leq \delta_q.\end{align*}
It follows that
\begin{align*}\lefteqn{\vert\skalp{(z_1,z_2),(y_1,y_2)}-\skalp{(z_1^*,z_2^*),(x,d)-(\xb,\xb)}\vert}\\
&=\vert\skalp{e,f(x)+d^*}-\skalp{c,x-d}-\skalp{\nabla f(x)^Te-c,x-\xb}-\skalp{e^*+c,d-\xb}\vert\\
&\leq\vert\skalp{e,f(x)-f(\xb)-\nabla f(x)(x-\xb)}\vert+\vert\skalp{e, d^*+f(\xb)}+\skalp{e^*,d-\xb}\vert\\
&\leq \frac\epsilon{2\sqrt{2}(l+1)}\norm{e}\norm{x-\xb}+\frac \epsilon{2\sqrt{2}(l+1)}\norm{(e,e^*)}\norm{(d-\xb, d^*+f(\xb))}\\
&\leq\frac\epsilon{\sqrt{2}(l+1)}\norm{(e,e^*)}\norm{(x-\xb,d-\xb,d^*+f(\xb))}\\
&\leq \frac\epsilon{\sqrt{2}(l+1)}\norm{(e,e^*)}\big(\norm{(x-\xb,d-\xb,d^*+f(x))}+\norm{f(x)-f(\xb)}\big)\\
&\leq \frac \epsilon {\sqrt{2}}\norm{(e,e^*)}\norm{(x-\xb,d-\xb,d^*+f(x),x-d)}=\frac \epsilon{\sqrt{2}}\norm{(e,e^*)}\norm{((x,d),(y_1,y_2))-((\xb,\xb),(0,0))}.\end{align*}
Since $\min_{c}\norm{c}^2+\norm{e^*-c}^2=\frac 12\norm{e^*}^2$, we obtain $\norm{((z_1,z_2),(z_1^*,z_2^*))}^2\geq \norm{e}^2+\frac12\norm{e^*}^2\geq \frac 12\norm{(e,e^*)}^2$ and
\[\vert\skalp{(z_1,z_2),(y_1,y_2)}-\skalp{(z_1^*,z_2^*),(x,d)-(\xb,\xb)}\vert\leq \epsilon\norm{((z_1,z_2),(z_1^*,z_2^*))}\norm{((x,d),(y_1,y_2))-((\xb,\xb),(0,0))}.\]
Thus $\F$ is \SCD \ssstar at $((\xb,\xb),(0,0))$.
\end{proof}
We proceed now with the description of the approximation step. Given $(x^{(k)},d^{(k)})$ and a scaling parameter $\gamma^{(k)}$, we compute $u^{(k)}:=u_{\gamma^{(k)}}(x^{(k)})$  and  set
\begin{equation}
  \label{EqResApprStep1}\hat x^{(k)}= x^{(k)},\ \hat d^{(k)}=x^{(k)}+u^{(k)}\quad \mbox{and}\quad \hat y^{(k)}=(\hat y_1^{(k)},\hat y_2^{(k)})=-(\gamma^{(k)} u^{(k)},u^{(k)}).
\end{equation}
We observe that
\[((\hat x^{(k)},\hat d^{(k)}),(\hat y_1^{(k)},\hat y_2^{(k)}))\in \gph \F,\]
which follows immediately from the first-order optimality condition \eqref{EqOptCond_u}. Note that the outcome of the approximation step does not depend on the auxiliary variable $d^{(k)}$.  In order to apply Theorem \ref{ThConvSSNewton}, we have to show the existence of a real $\eta>0$ such that the estimate
\begin{equation}\label{EqEstApprStep1}
  \norm{((\hat x^{(k)}-\xb, \hat d^{(k)}-\xb), \hat y^{(k)}}\leq \eta\norm{(x^{(k)}-\xb,d^{(k)}-\xb)},
\end{equation}
corresponding to \eqref{EqBndApprStep}, holds for all $(x^{(k)},d^{(k)})$ with $x^{(k)}$ close to $\xb$. By virtue of \eqref{EqResApprStep1} the left-hand side of \eqref{EqEstApprStep1} amounts to
\begin{align}\nonumber
  \norm{((\hat x^{(k)}-\xb, \hat x^{(k)}+ u^{(k)}-\xb),(-\gamma^{(k)} u^{(k)}, -u^{(k)})}&\leq \norm{(\hat x^{(k)}-\xb, \hat x^{(k)}-\xb,0,0)}+\norm{(0, u^{(k)},-\gamma^{(k)} u^{(k)}, -u^{(k)})}\\
  \label{EqAuxBnd1}&\leq  2\norm{\hat x^{(k)}-\xb}+(2+\gamma^{(k)})\norm{u^{(k)}}.
\end{align}
Since $u_{\gamma^{(k)}}(\xb)=0$, we obtain from \eqref{EqLip2} the bounds
\begin{align*}&\norm{\hat d^{(k)}-\xb}\leq \norm{x^{(k)}-\xb}+\frac 1{\gamma^{(k)}}\norm{f(x^{(k)})-f(\xb)}\\
&\norm{u^{(k)}}\leq 2\norm{x^{(k)}-\xb}+\frac 1{\gamma^{(k)}}\norm{f(x^{(k)})-f(\xb)}.
\end{align*}
The latter estimate, together with \eqref{EqAuxBnd1}, imply
\begin{align}\nonumber\norm{((\hat x^{(k)}-\xb, \hat d^{(k)}-\xb), \hat y^{(k)}}&\leq \Big(2+(2+\gamma^{(k)})\big(2+ \frac l{\gamma^{(k)}}\big)\Big)\norm{x^{(k)}-\xb}\\
\label{EqAppr1L}&\leq  \Big(2+(2+\gamma^{(k)})\big(2+ \frac l{\gamma^{(k)}}\big)\Big)\norm{(x^{(k)}-\xb,d^{(k)}-\xb)},\end{align}
where $l$ is the Lipschitz constant of $f$ on a neighborhood of $\xb$. Thus the desired inequality \eqref{EqEstApprStep1} holds, as long as $\gamma^{(k)}$ remains bounded and bounded away from zero.

Next we describe the Newton step. According to Algorithm \ref{AlgNewton} and \eqref{EqSCD_F2}, we have to compute a pair ${Y^*}^{(k)},{X^*}^{(k)}$ with $\rge({Y^*}^{(k)},{X^*}^{(k)})\in\Sp^*\partial q(\hat d^{(k)},\hat d^*{}^{(k)}) $ and then to solve the linear system
\begin{equation*}
   \left(\begin{matrix}{{Y^*}^{(k)}}^T\nabla f(x^{(k)})&{{X^*}^{(k)}}^T\\-I&I\end{matrix}\right)\myvec{\Delta x^{(k)}\\\Delta d^{(k)}}=-\left(\begin{matrix}{{Y^*}^{(k)}}^T&0\\0&-I\end{matrix}\right)\myvec{\hat y_1^{(k)}\\\hat y_2^{(k)}}
\end{equation*}
Simple algebraic transformations yield
\begin{equation}\label{EqNewtonSyst}({{Y^*}^{(k)}}^T\nabla f(x^{(k)})+{{X^*}^{(k)}}^T)\Delta x^{(k)}= -({{Y^*}^{(k)}}^T\hat y_1^{(k)}+{{X^*}^{(k)}}^T\hat y_2^{(k)})
\end{equation}
and $\Delta d^{(k)}=\hat y_2^{(k)}+\Delta x^{(k)}$. Using \eqref{EqResApprStep1} the system \eqref{EqNewtonSyst} amounts to
\begin{equation}\label{EqNewtonSyst1}({{Y^*}^{(k)}}^T\nabla f(x)+{{X^*}^{(k)}}^T)\Delta x^{(k)}= (\gamma^{(k)}{{Y^*}^{(k)}}^T+{{X^*}^{(k)}}^T)u^{(k)}.\end{equation}
Having computed the Newton direction, the new iterate is given by $x^{(k+1)}=x^{(k)}+\Delta x^{(k)}=d^{(k+1)}$. We summarize our considerations in the following algorithm, where the auxiliary variable $d^{(k)}$ is omitted.
\begin{algorithm}[\ssstar Newton Method for VI of the second kind \eqref{EqVI2ndK}]\label{AlgSSNewtVIBasic}\mbox{ }\\
1. Choose starting point $x^{(0)}$ and set the iteration counter $k:=0$.\\
2. If $0\in H(x^{(k)})$ stop the algorithm.\\
3. \begin{minipage}[t]{\myAlgBox}Select a parameter $\gamma^{(k)}>0$, compute $u^{(k)}:=u_{\gamma^{(k)}}(x^{(k)})$ and set $\hat d^{(k)}:=x^{(k)}+u^{(k)}$, ${\hat d^*}{}^{(k)}:=-\gamma^{(k)}u^{(k)}-f(x^{(k)})$.\end{minipage}\\
4. \begin{minipage}[t]{\myAlgBox} Select $({X^*}^{(k)}, {Y^*}^{(k)})$ with $\rge({Y^*}^{(k)}, {X^*}^{(k)} )\in\Sp^* \partial q((\hat d^{(k)},\hat d^*{}^{(k)})$, compute the Newton direction $\Delta x^{(k)}$ from  \eqref{EqNewtonSyst1} and set $x^{(k+1)}=x^{(k)}+\Delta x^{(k)}$.\end{minipage}\\
5. Increase the iteration counter $k:=k+1$ and go to Step 2.
\end{algorithm}
Combining Theorem \ref{ThConvSSNewton} with Proposition \ref{PropF} we obtain the following convergence result.
\begin{theorem}\label{ThConvSSNewtonVI}Let $\xb\in H^{-1}(0)$ be a solution of \eqref{EqVI2ndK} and assume that $\partial q$ is \SCD \ssstar at $(\xb,-f(\xb)$. Further suppose that $H$ is \SCD regular around $(\xb,0)$. Then for every pair $\underline{\gamma},\bar\gamma$ with $0<\underline{\gamma}\leq \bar\gamma$ there exists a neighborhood $U$ of $\xb$ such that for every starting point $x^{(0)}\in U$ Algorithm \ref{AlgSSNewtVIBasic} produces a sequence $x^{(k)}$ converging superlinearly to $\xb$, provided we choose in every iteration step $\gamma^{(k)}\in [\underline{\gamma},\bar\gamma]$.
\end{theorem}
\section{Globalization}
In the preceding section we showed locally superlinear convergence of our implementation of the semismooth* Newton method. However, we do not only want fast local convergence but also convergence from arbitrary starting points. To this end we consider  a non-monotone line-search heuristic as well as hybrid approaches which combine this heuristic with some globally convergent method.

To perform the line search we need some merit function. Similar to the damped Newton method for solving smooth equations, we use some kind of residual. Here we  define the residual by means of the approximation step, i.e., given $x$ and $\gamma>0$, we use
\begin{align}
  \label{EqRes1}
  r_\gamma(x):=\norm{(\hat y_1^{(k)},\hat y_2^{(k)})}=\norm{(\gamma u_\gamma(x),u_\gamma(x))}=\sqrt{1+\gamma^2}\norm{u_\gamma(x)}
\end{align}
as motivated by \eqref{EqResApprStep1}. Note that every evaluation of the residual function $r_\gamma(x)$ requires the computation of $u_\gamma(x)$.

Our globalization approaches are intended mainly for the case when the variational inequality \eqref{EqVI2ndK} does not correspond to the solution of some nonsmooth optimization problem. For the solution of optimization problems, namely, there exist  more efficient globalization strategies based on merit functions derived from the objective and this case will be treated in a forthcoming paper.

\subsection{\label{SecHeuristic}A non-monotone line-search heuristic}
In general, we replace the full Newton step 4. in Algorithm \ref{AlgSSNewtVIBasic} by a damped step of the form
\[x^{(k+1)}=\hat x^{(k)}+\alpha^{(k)}\triangle x^{(k)}, \]
where $\alpha^{(k)}\in(0,1]$ is chosen such that the line search condition
\begin{equation}\label{EqLineSearch}r_{\gamma^{(k)}}(\hat x^{(k)}+\alpha^{(k)}s^{(k)})\leq (1+\delta^{(k)}-\mu \alpha^{(k)})r_{\gamma^{(k)}}(\hat x^{(k)})\end{equation}
is fulfilled, where $\mu\in(0,1)$ and $\delta^{(k)}$ is a given sequence of positive numbers converging to $0$.

Obviously, the step size $\alpha^{(k)}$ exists since the residual function $r_\gamma(x)$ is continuous. However, it is not guaranteed that the residual is decreasing, i.e., that $r_{\gamma^{(k)}}(x^{(k+1)})<r_{\gamma^{(k)}}(\hat x^{(k)})$.

The computation of $\alpha^{(k)}$ can be done in the usual way. For instance, we can choose the first element of a sequence $(\beta_j)$, which fulfills  $\beta_0=1$ and converges monotonically to zero, such that the line search condition \eqref{EqLineSearch} is fulfilled.

 For $\gamma^{(k)}$ we suggest a choice with $\gamma^{(k)}\approx\norm{\nabla f(x^{(k)}}$. Since the spectral norm $\norm{\nabla f(x^{(k)}}$ is difficult to compute, we use an easy computable norm instead, e.g., the  maximum absolute column sum norm $\norm{\nabla f(x^{(k)})}_1$.

 Although we are not able to show convergence properties for this heuristic, it showed good convergence properties in practice.

 \subsection{
 Globally convergent hybrid approaches}

 In this subsection we suggest a combination of the \ssstar Newton method with some existing globally convergent method which exhibits both global convergence and local superlinear convergence. Assume that the used globally convergent method is formally given by some mapping $\T:\R^n\to\R^n$, which computes from some iterate $x^{(k)}$ the next iterate by
 \[x^{(k+1)}=\T(x^{(k)}).\]
 Of course, $\T$ must depend on the problem \eqref{EqVI2ndK} which we want to solve and will presumably depend also on some additional parameters which control the behavior of the method. In our notation we neglect to a large extent these dependencies.

 Consider the following well-known examples for such a mapping $\T$.
 \begin{enumerate}
   \item For the forward-backward splitting method, the mapping $\T$ is given by
   \begin{equation*}
     \T^{\rm FB}_\lambda(x)=(I+\lambda\partial q)^{-1}(I-\lambda f)(x),
   \end{equation*}
   where $\lambda>0$ is a suitable prarameter. Note that $\T^{\rm FB}_\lambda(x)=x+u_{1/\lambda}(x)$.
   \item For the Douglas-Rachford splitting method we  have
   \begin{equation*}
     \T^{\rm DR}_\lambda(x)=(I+\lambda f)^{-1}\Big((I+\lambda\partial q)^{-1}(I-\lambda f) +\lambda f\Big)(x)=(I+\lambda f)^{-1}(\T^{\rm FB}_\lambda+\lambda f)(x),
   \end{equation*}
   where $\lambda>0$ is again some parameter.
   \item A third method is given by the hybrid projection-proximal point algorithm due to Solodov and Svaiter \cite{SolSv99}. Let $x$ and $\gamma>0$ be given and consider $\hat x=\T^{\rm FB}_{1/\gamma}(x)$, i.e. $\hat x-x=u_{\gamma}(x)$. Then $0\in \gamma(\hat x-x)+f(x)+\partial q(\hat x)$ and consequently
       \begin{equation*}
         0\in v + \gamma(\hat x-x)+(f(x)-f(\hat x)),
       \end{equation*}
       where $v:=-\gamma(\hat x-x)+f(\hat x)-f(x)\in H(\hat x).$
   Then, in the hybrid projection-proximal point algorithm the mapping $\T$ is given by the projection of $x$ on the hyperplane $\{z\mv \skalp{v,z-\hat x}=0\}$, i.e.,
   \begin{equation*}
   \T^{\rm PM}_\gamma(x)=x-\frac{\skalp{v,\hat x-x}}{\norm{v}^2}v.
   \end{equation*}
 \end{enumerate}
 Note that in principle we could also use other methods which depend not only on the last iterate like the golden ratio algorithm \cite{Mal20}, but for ease of presentation  these methods are omitted.
 \begin{algorithm}[Globally convergent hybrid \ssstar Newton method for VI of the second kind]\label{AlgSSNewt_Hybrid}\mbox{ }\\
Input: A method for solving \eqref{EqVI2ndK} given by the iteration operator $\T:\R^n\to\R^n$, a starting point $x^{(0)}$,  line search parameter $0<\nu<1$, a sequence $\delta^{(k)}\in (0,1)$, a sequence $\beta_j\downarrow 0$ with $\beta_0=1$  and a stopping tolerance $\epsilon_{tol}>0$.\\
1. Choose $\gamma^{(0)}$, set $r_N^{(0)}:= r_{\gamma^{(0)}}( x^{(0)})$ and set the counters $k:=0$, $l:=0$.\\
2. If $r_{\gamma^{(k)}}(x^{(k)})\leq\epsilon_{tol}$ stop the algorithm.\\
3. \parbox[t]{\myAlgBox}{Perform the approximation step as in Algorithm \ref{AlgSSNewtVIBasic} and compute  the Newton direction $\Delta x^{(k)}$ by solving \eqref{EqNewtonSyst1}. Try to determine the step size $\alpha^{(k)}$ as the first element from the sequence $\beta_j$ satisfying $\beta_j>\delta^{(l)}$ and
\[r_{\gamma^{(k)}}(x^{(k)}+\beta_j\Delta x^{(k)})\leq (1-\nu \beta_j) r_N^{(l)}.\]}
4. If both $\Delta x^{(k)}$ and $\alpha^{(k)}$ exist, set $x^{(k+1)}=x^{(k)}+ \alpha^{(k)} \Delta x^{(k)}$, $r_N^{(l+1)} =r_{\gamma^{(k)}}(x^{(k+1)})$ and increase $l:=l+1$.\\
5. Otherwise, if the Newton direction $\Delta x^{(k)}$ or the step length  $\alpha^{(k)}$ does not exist, compute $x^{(k+1)}=\T(x^{(k)})$.\\
6. Update $\gamma^{(k+1)}$ and increase the iteration counter $k:=k+1$ and go to Step 2.
\end{algorithm}
In what follows we denote by $k_l$ the subsequence of iterations where the new iterate $x^{k+1}$ is computed by the damped Newton Step 4, i.e.,
\[x^{(k_l)}= x^{(k_l-1)}+\alpha^{(k_l-1)}\Delta x^{(k_l-1)},\ r_N^{(l)}=r_{\gamma^{(k_l-1)}}(x^{(k_l)}).\]
\begin{theorem}
  Assume that the GE \eqref{EqVI2ndK} has at least one solution and assume that the solution method given by the iteration mapping $\T:\R^n\to\R^n$ has the property that for every starting point $y^{(0)}\in\R^n$ the sequence $y^{(k)}$, given by the recursion $y^{(k+1)}=\T(y^{(k)})$, has at least one accumulation point which is a solution to the GE \eqref{EqVI2ndK}. Then for every starting point $x^{(0)}$ the sequence $x^{(k)}$ produced by Algorithm \ref{AlgSSNewt_Hybrid} with $\epsilon_{tol}=0$ and $\sum_{k=0}^\infty \delta^{(k)}=\infty$ has the following properties.
  \begin{enumerate}
  \item[(i)] If the Newton step is accepted only finitely many times in step 4, then the sequence $x^{(k)}$ has at least one accumulation point which solves \eqref{EqVI2ndK}.
  \item[(ii)] If the Newton step is accepted infinitely many times in step 4, then every accumulation point of the subsequence $x^{(k_l)}$ is a solution to \eqref{EqVI2ndK}.
  \item[(iii)] If there exists an accumulation point $\xb$ of the sequence $x^{(k)}$ which solves \eqref{EqVI2ndK}, the mapping $H$ is  \SCD regular around $(\xb,0)$ and $\partial q$ is  \SCD \ssstar at $(\xb,-f(\xb))$, then the sequence $x^{(k)}$ converges superlinearly to $\xb$ and the Newton step in step 3 is accepted with step length $\alpha^{(k)}=1$ for all $k$ sufficiently large, provided the sequence $\gamma^{(k)}$ satisfies
      \begin{equation*}
      0<\underline{\gamma}\leq \gamma^{(k)}\leq\bar\gamma\ \forall k
      \end{equation*}
       for some positive reals $\underline{\gamma},\bar\gamma$.
  \end{enumerate}
\end{theorem}
\begin{proof}
  The first statement is an immediate consequence of our assumption on $\T$. In order to show the second statement, observe that the sequence $r_N^{(l)}$ satisfies
  $r_N^{(l+1)}\leq (1-\nu\delta^{(l)})r_N^{(l)}$ implying
  \[\lim_{l\to\infty} \ln(r_N^{(l+1)})-\ln(r_N^{(0)}) \leq\lim_{l\to\infty}\sum_{i=0}^l\ln(1-\nu\delta^{(i)})\leq -\lim_{l\to\infty}\sum_{i=0}^l\nu \delta^{(i)}=-\infty.\]
  Thus $\lim_{l\to\infty}r_N^{(l)}=\lim_{l\to\infty}\sqrt{1+{\gamma^{(k_l-1)}}^2}\norm{u_{\gamma^{(k_l-1)}}(x^{(k_l)})}=0$ and we can conclude that \[\lim_{l\to\infty}\norm{u_{\gamma^{(k_l-1)}}(x^{(k_l)})}=\lim_{l\to\infty}\gamma^{(k_l-1)}\norm{u_{\gamma^{(k_l-1)}}(x^{(k_l)})}=0.\]
   Together with the inclusion
  \[0\in \gamma^{(k_l-1)}u_{\gamma^{(k_l-1)}}(x^{(k_l)})+f(x^{(k_l)})+ \partial q(x^{(k_l)}+u_{\gamma^{(k_l-1)}}(x^{(k_l)})),\]
  the continuity of $f$ and the closedness of $\gph\partial q$, it follows that $0\in f(\xb)+\partial q(\xb)$ holds for every accumulation point $\xb$ of the subsequence $x^{(k_l)}$. This proves our second assertion.

  Finally we want to show (iii). Assume that $\xb$ is an accumulation point of the sequence $x^{(k)}$ such that the mapping $H$ is \SCD regular around $(\xb,0)$ and $\partial q$ is  \SCD \ssstar at $(\xb,-f(\xb))$.  By Proposition \ref{PropF} the mapping $\F$ is \SCD regular and \SCD \ssstar at $((\xb,\xb),(0,0))$. By invoking \cite[Theorem 6.2]{GfrOut21a}, the mapping $\F$ is strongly metrically subregular at $((\xb,\xb),(0,0))$ and, moreover, there is some $\kappa>0$ and some neighborhoods $U$ of $(\xb,\xb)$ and $V$ of $(0,0)$ such that
  \begin{align}\label{EqAuxStrMetrSubr}&\norm{(x,d)-(\xb,\xb)}\leq \kappa \dist{(0,0),\F(x,d)}\ \forall (x,d)\in U,\\
  \label{EqAuxSCDReg}&L\in\Z_{2n}^{\rm reg}\mbox{ and }\norm{C_L}\leq \kappa\ \forall L\in\Sp^*\F ((x,d),(y_1,y_2))\ \forall ((x,d),(y_1,y_2)) \in U\times V.
  \end{align}
  Thus, whenever $((\hat x^{(k)},\hat d^{(k)}),\hat y^{(k)})\in U\times V$, the Newton direction $(\Delta x^{(k)},\Delta d^{(k)})$ exists and satisfies $\norm{(\Delta x^{(k)},\Delta d^{(k)}}\leq \kappa\norm{\hat y^{(k)}}$.

  By Proposition \ref{PropConvNewton} and \eqref{EqAuxSCDReg}, for every $\epsilon>0$ there is some $\delta>0$ such that
  \begin{align}\nonumber\norm{\hat x^{(k)}+\Delta x^{(k)}-\xb}&\leq \norm{\myvec{\hat x^{(k)}+\Delta x^{(k)}-\xb\\\hat d^{(k)}+\Delta d^{(k)}-\xb}}\leq \epsilon \sqrt{2n(1+\kappa^2)}\norm{((\hat x^{(k)}-\xb,\hat d^{(k)}-\xb),\hat y^{(k)})}
  \end{align}
  whenever $((\hat x^{(k)},\hat d^{(k)}),\hat y^{(k)})\in \B_\delta((\xb,\xb),(0,0))$. Thus we  can find some $\delta'\in(0,1]$ such that $\B_{\delta'}((\xb,\xb),(0,0))\subset U\times V$ and
  \[\norm{\hat x^{(k)}+\Delta x^{(k)}-\xb}\leq \min\Big\{\frac{1-\nu}{c_1c_2\kappa \sqrt{1+
  \bar\gamma^2}},\frac 1{2c_2}\Big\}\norm{((\hat x^{(k)}-\xb,\hat d^{(k)}-\xb),\hat y^{(k)})}\]
  for $((\hat x^{(k)},\hat d^{(k)}),\hat y^{(k)})\in \B_{\delta'}((\xb,\xb),(0,0)$, where $c_1:=2+\frac l{\underline{\gamma}}$, $c_2:=2+(2+\bar\gamma) c_1$ and $l$ is some Lipschitz constant of $f$ in $\B_1(\xb)$. From \eqref{EqAppr1L} we deduce $\norm{((\hat x^{(k)}-\xb,\hat d^{(k)}-\xb),\hat y^{(k)})}\leq c_2\norm{x^{(k)}-\xb}$ yielding
  \begin{align}\label{EqAuxBnd2}
    \norm{\hat x^{(k)}+\Delta x^{(k)}-\xb}\leq \min\Big\{\frac{1-\nu}{c_1\kappa\sqrt{1+
  \bar\gamma^2}},\frac 12\Big\}\norm{\hat x^{(k)}-\xb}
  \end{align}
  for $x^{(k)}\in\B_{\bar \delta}(\xb)$ with $\bar\delta:=\delta'/c_2$.
  We now claim that for every iterate $x^{(k)}\in \B_{\bar\delta}(\xb)$ the Newton step  with step size $\alpha^{(k)}=1$ is accepted.
  If $x^{(k)}\in \B_{\bar\delta}(\xb)$ then $((\hat x^{(k)},\hat d^{(k)}),\hat y^{(k)})\in\B_{\delta'}((\xb,\xb),(0,0))\subset U\times V$ and from \eqref{EqAuxStrMetrSubr} we obtain
  \begin{equation*}\norm{x^{(k)}-\xb}\leq \norm{(\hat x^{(k)}, \hat d^{(k)})-(\xb,\xb)}\leq \kappa \dist{(0,0),\F(\hat x^{(k)}, \hat d^{(k)})}\leq\kappa\norm{\hat y^{(k)}}\leq \kappa\sqrt{1+\bar\gamma^2}\norm{u^{(k)}}.\end{equation*}
  Since $u_{\gamma^{(k)}}(\bar x)=0$, we obtain from \eqref{EqLip2} and \eqref{EqAuxBnd2} that
  \begin{align*}
  \norm{u_{\gamma^{(k)}}(x^{(k)}+\Delta x^{(k)})}&\leq c_1\norm{x^{(k)}+\Delta x^{(k)}-\xb}\leq \frac{1-\nu}{\kappa\sqrt{1+\bar\gamma^2})}\norm{x^{(k)}-\xb}\leq (1-\nu)\norm{u^{(k)}}\\
  &=(1-\nu)\norm{u_{\gamma^{(k)}}(x^{(k)})}
  \end{align*}
  showing
  \[ r_{\gamma^{(k)}}(x^{(k)}+\Delta x^{(k)})=\sqrt{1+{\gamma^{(k)}}^2}\norm{u_{\gamma^{(k)}}(x^{(k)}+\Delta x^{(k)})}\leq (1-\nu) \sqrt{1+{\gamma^{(k)}}^2}\norm{u_{\gamma^{(k)}}(x^{(k)})}=(1-\nu)r_{\gamma^{(k)}}(x^{(k)}).\]
  From this we conclude that the step size $\alpha^{(k)}=1$ is accepted.
  Now let $\bar k$ denote the first index such that $x^{(\bar k)}$ enters the ball $\B_{\bar\delta}$. Then for all $k\geq \bar k$ we have
  \[x^{(k+1)}=x^{(k)}+\Delta x^{(k)}, \quad \norm{x^{(k+1)}-\xb}\leq \frac 12\norm{x^{(k)}-\xb}\]
  and superlinear convergence follows from Theorem \ref{ThConvSSNewtonVI}.
\end{proof}

\section{Numerical Experiments}
Based on the general results from \cite{Fl}, the authors in \cite{OV} considered an evolutionary  Cournot-Nash equilibrium, where in the course of time the players (producers) adjust their productions to respond adequately to changing external parameters. Following \cite{Fl}, however, each change of production is generally associated with some expenses, called {\em costs of change}. In this way one obtains a generalized equation \eqref{EqVI2ndK} which has to be solved repeatedly in each
selected time step.

In this paper  we make the model from \cite{OV} more involved by admitting multiple commodities  and more realistic production constraints. As the solver of the respective generalized equation \eqref{EqVI2ndK}, the \SCD \ssstar Newton method (Algorithm \ref{AlgSSNewtVIBasic}) will be employed. The new model is described  as follows:
Let $n,m$ be the number of players and the number of produced commodities, respectively. Further, let $x=(x^1,\ldots,x^n)\in (\R^m)^n$ be the cumulative vector of productions, where
\[x^i=(x_1^i,x_2^i,\ldots, x_m^i)\in\R^m_+,\quad i=1,2,\ldots,n\]
stands for the {\em production portfolio} of the $i$-th player. With each player we associate
\begin{itemize}
  \item the mapping $c^i:\R^m_+\to\R$ which assigns  $x^i$ the respective  production cost;
  \item the linear system of inequalities $\Xi^ix^i\leq \zeta^i$ with a $p^i\times m$ matrix $\Xi^i$ and a vector $\zeta^i\in \R^{p^i}$ which specifies the set of {\em feasible productions} $\Omega^i=\{x^i\in\R^m\mv \Xi^ix^i\leq \zeta^i\}\subseteq \R^m_+$, and
  \item the cost of change $z^i:\R^m\to\R$ which assigns each change of the production portfolio $\triangle x^i\in\R^m$ the corresponding cost.
\end{itemize}
Clearly, the vector $t=(t_1,t_2,\ldots,t_m)$ with $t_j=\sum_{i=1}^n x_j^i$, $j=1,\ldots,m$, provides the overall amounts of single commodities which are available on the market in the considered time period. The price of the $j$-th commodity is given via the respective {\em inverse demand function} $\pi_j:\R_+\to\R_+$ assigning each value $t_j$ the corresponding price, at which the consumers are willing to buy.

Putting everything together, one arrives at the GE \eqref{EqVI2ndK}, where
\[f(x)=\left(\begin{matrix}
  f^1(x)\\\vdots\\f^n(x)
\end{matrix}\right)\quad \mbox{with}\quad f^i(x)=\nabla c^i(x^i)-\myvec{\pi_1(t_1)\\\vdots\\\pi_m(t_m)}-\myvec{x^i_1\nabla\pi_1(t_1)\\\vdots\\x_m^i\nabla\pi_m(t_m)}\]
and $q(x)=\sum_{i=1}^n\big(z^i(x^i)+\delta_{\Omega^i}(x^i)\big)$, $i=1,2,\ldots,n$. Concerning functions $c^i$, $i=1,\ldots,n$, and $\pi_j$, $j=1,\ldots,m$, we use functions of the same type as in \cite{MSS}, i.e.,
\begin{equation}\label{EqOligo1}c^i(x^i)=\sum_{j=1}^m\Big(b_j^ix^i_j+\frac{\delta_j^i}{\delta_j^i+1}{K_j^i}^{-\frac1{\delta_j^i}}\vert x_j^i \vert^{\frac{\delta_j^i+1}{\delta_j^i}}\Big),\ i=1,\ldots,n\end{equation}
with positive parameters $b_j^i$, $\delta_j^i$ and $K_j^i$, and
\begin{equation}\label{EqOligo2}\pi_j(t_j)=(1000n)^{\frac 1{\gamma_j}}t_j^{-\frac 1{\gamma_j}},\ j=1,\ldots,m\end{equation}
with positive parameters $\gamma_j$.

The functions $z^i$ are modeled in the form
\begin{equation}\label{EqOligo3}
z^i(\triangle x^i)=z^i(x^i- a^i)=\sum_{j=1}^m\beta_j^i\vert x_j^i-a_j^i\vert,\ i=1,\ldots,n,\end{equation}
where $a^i\in\Omega^i$ signifies the ''previous'' production portfolio of the $i$-th player and and the weights $\beta_j^i$ are positive reals indicating the costs of a ''unit'' change of production of the $j$-th commodity by the $i$-th player.

On the basis of \cite{OV} and \cite{MSS} it can be shown that for each fixed choice of the parameters in \eqref{EqOligo1},\eqref{EqOligo2} and \eqref{EqOligo3} the mapping $H(x)=f(x)+\partial q(x)$ is strictly monotone and the respective GE \eqref{EqVI2ndK} has a unique solution $\xb$ such that $H$ is strongly metrically regular around $(\xb,0)$.
From Theorem \ref{ThSubdiffConv} and \cite[Proposition 3.15]{GfrOut21a} it follows that $H$ is an \SCD mapping whenever $f$ is continuously differentiable near $\xb$. Consequently, since $\gph \partial q$ is a polyhedral mapping, we infer from \cite[Propositions 3.5, 3.6, 3.7]{GfrOut21} that in such a situation $H$ is SCD \ssstar and so the conceptual Algorithm \ref{AlgNewton} may be used. However, when implementing Algorithm \ref{AlgSSNewtVIBasic}, one has to be careful because the mapping $f$ does not meet the requirement of continuous differentiability on $\R^n$.
Therefore we replace $\pi_j$ by the twice continuously differentiable functions
\[\hat \pi_j(t_j):=\begin{cases}\pi_j(t_j)&\mbox{if $t_j>\epsilon_1$}\\\pi_j(\epsilon_1)+\pi_j'(\epsilon_1)(t_j-\epsilon_1)+\frac 12\pi_j''(\epsilon_1)(t_j-\epsilon_1)^2&\mbox{if $t_j\leq\epsilon_1$}\end{cases}\]
and, in the definition of $c^i(x^i)$, we replace the term $\vert x_j^i\vert$ by $\sqrt{(x_j^i)^2+\epsilon_2^2}$ whenever $\delta_j^i<1$ (in our implementation we used $\epsilon_1:=10^{-1}, \epsilon_2:=10^{-10}$). Since the functions $c^i$ are convex, one could alternatively  incorporate them in $q$  without smoothing instead of treating them as part of $f$.

Next we describe the approximation step of Algorithm \ref{AlgSSNewtVIBasic}, where $x^{(k)}=\big((x^1)^{(k)},\ldots,(x^n)^{(k)}\big)$ stands for the $k$-th iterate. For a given scaling parameter $\gamma^{(k)}>0$ and $i=1,2,\ldots,n$ we compute consecutively the (unique) solutions $(u^i)^{(k)}$ $,i=1,\ldots,n$ of the strictly convex optimization problems
\begin{equation}\label{EqOligoAppr1}\min_{u^i\in\R^m}\frac {\gamma^{(k)}}2\norm{u^i}^2+\skalp{f^i(x^{(k)}),u^i}+q^i\big((x^i)^{(k)}+u^i\big),\end{equation}
obtaining thus the vector $u^{(k)}=\big((u^1)^{(k)},\ldots,(u^n)^{(k)}\big)\in(\R^m)^n$. Due to the specific structure of the functions $q^i$, problem \eqref{EqOligoAppr1} can be replaced by the standard quadratic program
\begin{align*}
\nonumber  \min_{(u^i,v^i)\in\R^m\times\R^m}&\frac{\gamma^{(k)}}2\norm{u^i}^2+\skalp{f^i(x^{(k)}),u^i}+\sum_{j=1}^m \beta^i_j v^i_j\\
  \mbox{subject to}\quad& \Xi^i\big((x^i)^{(k)}+u^i\big)\leq \zeta^i\\
\nonumber  &\left.\begin{array}{l}v_j^i\geq \quad \, (x^i_j)^{(k)}+u^i_j-a^i_j\\
  v_j^i\geq -\big((x^i_j)^{(k)}+u^i_j-a^i_j\big)\end{array}\right\}j=1,\ldots,m.
\end{align*}
Clearly, the u-component of the solution amounts exactly to the (unique) solution of \eqref{EqOligoAppr1}. The outcome of the projection step is then given by the update \eqref{EqResApprStep1}, i.e.,
\[\hat x^{(k)}= x^{(k)},\ \hat d^{(k)}=x^{(k)}+u^{(k)}=\big((\hat d^1)^{(k)},\ldots,(\hat d^n)^{(k)}\big)\quad \mbox{and}\quad \hat y^{(k)}=-(\gamma^{(k)} u^{(k)},u^{(k)}).\]
In the Newton step we make use of the following theorem.
\begin{theorem}\label{ThSubSpOligo}
  Let $g:\R^m\to\bar\R$ be given by $g(x)=\sum_{j=1}^m\beta_j\vert x_j-a_j\vert+\delta_\Omega(x)$, where $\beta_j\geq0$, $a_j\in\R$, $j=1,\ldots,m$ and $\Omega=\{x\in\R^n\mv \skalp{\xi_l,x}\leq \zeta_l,\ l=1,\ldots,p\}$ is a convex polyhedral set given by the vectors $\xi_l\in\R^m$ and scalars $\zeta_l\in\R$, $l=1,\ldots,p$. Then for every $(x,x^*)\in\gph \partial g$ there holds
  \[W(x)\times W(x)^\perp\in \Sp\partial g(x,x^*)=\Sp^*\partial g(x,x^*),\]
  where $W(x):=\{w\in\R^m \mv w_i=0,\ i\in J_0(x), \skalp{\xi_l,w}=0, l\in L(x)\}$ with $J_0(x):=\{j\mv \beta_j>0,\ x_j=a_j\}$ and $L(x)=\{l\mv\skalp{\xi_l,x}=\zeta_l\}$.
\end{theorem}
\begin{proof}
  By standard calculus rules of convex analysis, for every $x\in\Omega$ we have
  \begin{align*}\lefteqn{\partial g(x)=N_\Omega(x)+\sum_{j:\beta_j>0}\beta_j\partial \vert x_j-a_j\vert}\\
  &=\{\sum_{l\in L(x)}\xi_l\mu_l+\sum_{j\in J_+(x)}\beta_je_j-\sum_{j\in J_-(x)}\beta_je_j+\sum_{j\in J_0(x)1}^m\beta_j\tau_je_j\mv\mu_l\geq 0,\ l\in L(x), \tau_j\in[-1,1], j\in J_0(x)\},\end{align*}
  where $J_+(x):=\{j\mv \beta_j>0,\ x_j>a_j\}$, $J_-(x):=\{j\mv \beta_j>0,\ x_j<a_j\}$ and $e_j$ denotes the $j$-th unit vector.
  For every partition $J_0,J_+,J_-$ of $\{j\in\{1,\ldots,m\}\mv\beta_j>0\}$ and every index set $L\subseteq\{1,\ldots,p\}$ let
  \begin{align*}&D_{J_0,J_+,J_-,L}:=\Big\{x\,\Big\vert\, \begin{array}{l}
  x_j=a_j,\ j\in J_0,\ x_j\geq a_j,\ j\in J_+,\   x_j\leq a_j,\ j\in J_-\\
  \skalp{\xi_l,x}=\zeta_l,\ l\in L,\ \skalp{\xi_l,x}\leq \zeta_l,\ l\not\in L\end{array}\Big\},\\
  &\tilde D_{J_0,J_+,J_-,L}=\big\{\sum_{j\in J_0}\beta_j\tau_je_j+\sum_{j\in J_+}\beta_je_j-\sum_{j\in J_-}\beta_je_j+\sum_{l\in L}\xi_l\mu_l\mv
  \tau_j\in[-1,+1],\ j\in J_0,\ \mu_l\geq 0,\ l\in L\big\},\\
  &E_{J_0,J_+,J_-,L}:=D_{J_0,J_+,J_-,L}\times \tilde D_{J_0,J_+,J_-,L}.
  \end{align*}
Further we denote by $\I$ the collection of all those index sets $(J_0,J_+,J_-,L)$ such that
\[\ri D_{J_0,J_+,J_-,L}=\Big\{x\mv \begin{array}{l}
  x_j=a_j,\ j\in J_0,\  x_j> a_j,\ j\in J_+,\  x_j<a_j,\ j\in J_-\\
  \skalp{\xi_l,x}=\zeta_l,\ l\in L,\ \skalp{\xi_l,x}< \zeta_l,\ l\not\in L
  \end{array}\Big\}\not=\emptyset.\]
  It follows that for every $(J_0,J_+,J_-, L)\in\I$ and every $x\in D_{J_0,J_+,J_-,L}$ we have $x\in\Omega$ and $\tilde D_{J_0,J_+,J_-,L}\subseteq \partial g(x)$. Further, for every $x\in\Omega$ there holds $(J_0(x),J_+(x),J_-(x),L(x))\in\I$ and $\tilde D_{J_0(x),J_+(x),J_-(x),L(x)}= \partial g(x)$ implying
  \[\gph \partial g=\bigcup_{(J_0,J_+,J_-,L)\in\I}E_{J_0,J_+,J_-,L}.\]
  We now claim that for any two elements $(J_0,J_+,J_-,L)\not=(J_0',J_+',J_-',L')\in\I$ we have $E_{J_0',J_+',J_-',L'}\cap \ri E_{J_0,J_+,J_-,L}=\emptyset$.
  Note that $\ri E_{J_0,J_+,J_-,L}=\ri D_{J_0,J_+,J_-,L}\times \ri \tilde D_{J_0,J_+,J_-,L}$ and that
  \[\ri \tilde D_{J_0,J_+,J_-,L}=\big\{\sum_{j\in J_0}\beta_j\tau_je_j+\sum_{j\in J_+}\beta_je_j-\sum_{j\in J_-}\beta_je_j+\sum_{l\in L}\xi_l\mu_l\mv
  \tau_j\in(-1,+1),\ j\in J_0,\ \mu_l> 0,\ l\in L\big\}\]
  by \cite[Theorem 6.6.]{Ro70}.
  Assuming that this claim does not hold for some $(J_0,J_+,J_-,L)\not=(J_0',J_+',J_-',L')\in\I$, there are reals $\mu_l>0$, $l\in L$, $\mu_l'\geq 0$, $l\in L'$, $\tau_j\in(-1,1)$, $j\in J_0$, $\tau_j'\in [-1,1]$, $j\in J_0'$ such that
  \begin{equation}\label{EqAuxSumMu}\sum_{j\in J_0}\beta_j\tau_je_j+\sum_{j\in J_+}\beta_je_j-\sum_{j\in J_-}\beta_je_j+\sum_{l\in L}\xi_l\mu_l=\sum_{j\in J_0'}\beta_j\tau_j'e_j+\sum_{j\in J_+'}\beta_je_j-\sum_{j\in J_-'}\beta_je_j+\sum_{l\in L'}\xi_l\mu_l'\end{equation}
  and some $x\in D_{J_0',J_+',J_-',L'}\cap \ri D_{J_0,J_+,J_-,L}$ implying $J_0'\subseteq J_0$ and $L'\subseteq L$, where equality can not simultaneously hold in both inclusions.
  Choosing $x'\in\ri D_{J_0',J_+',J_-',L'}$ and setting $u=x'-x$, we obtain
  \[u_j=0,\ j\in J_0',\ u_j>0,\ j\in (J_0\setminus J_0')\cap J_+',\ u_j<0,\ j\in (J_0\setminus J_0')\cap J_-',\ \skalp{\xi_l,u}=0,\ l\in L', \skalp{\xi_l,u}<0,\ l\in L\setminus L'.\]
  Rearranging \eqref{EqAuxSumMu} yields
  \[\sum_{j\in (J_0\setminus J_0')\cap J_+'}\beta_j(\tau_j-1)e_j+\sum_{j\in (J_0\setminus J_0')\cap J_-'}\beta_j(\xi_j+1)e_j+\sum_{l\in L\setminus L'}\mu_l\xi_l=\sum_{j\in J_0'}\beta_j(\tau_j'-\tau_j)e_j+\sum_{l\in L'}(\mu_l'-\mu_l)\xi_l\]
  and by multiplying this equation with $u$ we obtain the contradiction
  \begin{align*}0&>\sum_{j\in (J_0\setminus J_0')\cap J_+'}\beta_j(\tau_j-1)u_j+\sum_{j\in (J_0\setminus J_0')\cap J_-'}\beta_j(\tau_j+1)u_j+\sum_{l\in L\setminus L'}\mu_l\skalp{\xi_l,u}\\
  &=\sum_{j\in J_0'}\beta_j(\tau_j'-\tau_j)u_j+\sum_{l\in L'}(\mu_l'-\mu_l)\skalp{\xi_l,u}=0.\end{align*}
  Hence, our claim holds true and we may conclude that for every $(J_0,J_+,J_-,L)\in\I$ and every $(z,z^*)\in \ri  E_{J_0,J_+,J_-,L}$ we have
  \begin{align*}\lefteqn{T_{\gph \partial g}(z,z^*)=T_{E_{J_0,J_+,J_-,L}}(z,z^*)=T_{D_{J_0,J_+,J_-,L}}(z)\times T_{\tilde D_{J_0,J_+,J_-,L}}(z^*)}\\
  &=\{w\mv w_j=0, j\in J_0, \skalp{\xi_l,w}=0, l\in L\}\times
  \{\sum_{j\in J_0}\beta_j\sigma_je_j+\sum_{l\in L}\xi_l\nu_l\mv \sigma_j\in\R,\ j\in J_0,\ \nu_l\in\R,\ l\in L\}\\
  &=  W(z)\times W(z)^\perp,\end{align*}
  where the last equality follows from $J_0=J_0(z)$ and $L=L(z)$.
  Now consider $(x,x^*)\in\gph \partial g$. Then $(J_0(x),J_+(x),J_-(x),L(x))\in \I$ and $x\in \ri D_{I_0(x),I_+(x),I_-(x),J(x)}$. Selecting $z^*\in \ri \tilde D_{J_0(x),J_+(x),J_-(x),L(x)}$, for all $\alpha\in(0,1]$ we have $x_\alpha^*:= (1-\alpha)x^*+\alpha z^*)\in \ri \tilde D_{J_0(x),J_+(x),J_-(x),L(x)}$ implying $T_{\gph \partial g}(x,x_\alpha^*)= W(x)\times W(x)^\perp$. Now the assertion follows from the definition of $\Sp\partial g(x,x^*)$ together with Theorem \ref{ThSubdiffConv}.
\end{proof}
\if{
\begin{theorem}\label{ThSubSpOligo}
  Let $g:\R^m\to\bar\R$ be given by $g(x)=\sum_{j=1}^m\beta_j\vert x_j-a_j\vert+\delta_\Omega(x)$, where $\beta_j>0$, $a_j\in\R$, $j=1,\ldots,m$ and $\Omega=\{x\in\R^n\mv \skalp{\xi_l,x}\leq \zeta_l,\ l=1,\ldots,p\}$ is a convex polyhedral set given by the vectors $\xi_l\in\R^m$ and scalars $\zeta_l\in\R$, $l=1,\ldots,p$. Then for every $(x,x^*)\in\gph \partial g$ there holds
  \[W(x)\times W(x)^\perp\in \Sp\partial g(x,x^*)=\Sp^*\partial g(x,x^*),\]
  where $W(x):=\{w\in\R^m \mv w_i=0,\ i\in J_0(x), \skalp{\xi_l,w}=0, l\in L(x)\}$ with $J_0(x):=\{j\mv x_j=a_j\}$ and $L(x)=\{l\mv\skalp{\xi_l,x}=\zeta_l\}$.
\end{theorem}
\begin{proof}
  By standard calculus rules of convex analysis, for every $x\in\Omega$ we have
  \begin{align*}\lefteqn{\partial g(x)=N_\Omega(x)+\sum_{j=1}^m\beta_j\partial \vert x_j-a_j\vert}\\
  &=\{\sum_{l\in L(x)}\xi_l\mu_l+\sum_{j\in J_+(x)}\beta_je_j-\sum_{j\in J_-(x)}\beta_je_j+\sum_{j\in J_0(x)1}^m\beta_j\tau_je_j\mv\mu_l\geq 0,\ l\in L(x), \tau_j\in[-1,1], j\in J_0(x)\},\end{align*}
  where $J_+(x):=\{j\mv x_j>a_j\}$, $J_-(x):=\{j\mv x_j<a_j\}$ and $e_j$ denotes the $j$-th unit vector.
  For every partition $J_0,J_+,J_-$ of $\{1,\ldots,m\}$ and every index set $L\subseteq\{1,\ldots,p\}$ let
  \begin{align*}&D_{J_0,J_+,J_-,L}:=\Big\{x\,\Big\vert\, \begin{array}{l}
  x_j=a_j,\ j\in J_0,\ x_j\geq a_j,\ j\in J_+,\   x_j\leq a_j,\ j\in J_-\\
  \skalp{\xi_l,x}=\zeta_l,\ l\in L,\ \skalp{\xi_l,x}\leq \zeta_l,\ l\not\in L\end{array}\Big\},\\
  &\tilde D_{J_0,J_+,J_-,L}=\big\{\sum_{j\in J_0}\beta_j\tau_je_j+\sum_{j\in J_+}\beta_je_j-\sum_{j\in J_-}\beta_je_j+\sum_{l\in L}\xi_l\mu_l\mv
  \tau_j\in[-1,+1],\ j\in J_0,\ \mu_l\geq 0,\ l\in L\big\},\\
  &E_{J_0,J_+,J_-,L}:=D_{J_0,J_+,J_-,L}\times \tilde D_{J_0,J_+,J_-,L}.
  \end{align*}
Further we denote by $\I$ the collection of all those index sets $(J_0,J_+,J_-,L)$ such that
\[\ri D_{J_0,J_+,J_-,L}=\Big\{x\mv \begin{array}{l}
  x_j=a_j,\ j\in J_0,\  x_j> a_j,\ j\in J_+,\  x_j<a_j,\ j\in J_-\\
  \skalp{\xi_l,x}=\zeta_l,\ l\in L,\ \skalp{\xi_l,x}< \zeta_l,\ l\not\in L
  \end{array}\Big\}\not=\emptyset.\]
  It follows that for every $(J_0,J_+,J_-, L)\in\I$ and every $x\in D_{J_0,J_+,J_-,L}$ we have $x\in\Omega$ and $\tilde D_{J_0,J_+,J_-,L}\subseteq \partial g(x)$. Further, for every $x\in\Omega$ there holds $(J_0(x),J_+(x),J_-(x),L(x))\in\I$ and $\tilde D_{J_0(x),J_+(x),J_-(x),L(x)}= \partial g(x)$ implying
  \[\gph \partial g=\bigcup_{(J_0,J_+,J_-,L)\in\I}E_{J_0,J_+,J_-,L}.\]
  We now claim that for any two elements $(J_0,J_+,J_-,L)\not=(J_0',J_+',J_-',L')\in\I$ we have $E_{J_0',J_+',J_-',L'}\cap \ri E_{J_0,J_+,J_-,L}=\emptyset$.
  Note that $\ri E_{J_0,J_+,J_-,L}=\ri D_{J_0,J_+,J_-,L}\times \ri \tilde D_{J_0,J_+,J_-,L}$ and that
  \[\ri \tilde D_{J_0,J_+,J_-,L}=\big\{\sum_{j\in J_0}\beta_j\tau_je_j+\sum_{j\in J_+}\beta_je_j-\sum_{j\in J_-}\beta_je_j+\sum_{l\in L}\xi_l\mu_l\mv
  \tau_j\in(-1,+1),\ j\in J_0,\ \mu_l> 0,\ l\in L\big\}\]
  by \cite[Theorem 6.6.]{Ro70}.
  Assuming that this claim does not hold for some $(J_0,J_+,J_-,L)\not=(J_0',J_+',J_-',L')\in\I$, there are reals $\mu_l>0$, $l\in L$, $\mu_l'\geq 0$, $l\in L'$, $\tau_j\in(-1,1)$, $j\in J_0$, $\tau_j'\in [-1,1]$, $j\in J_0'$ such that
  \begin{equation}\label{EqAuxSumMu}\sum_{j\in J_0}\beta_j\tau_je_j+\sum_{j\in J_+}\beta_je_j-\sum_{j\in J_-}\beta_je_j+\sum_{l\in L}\xi_l\mu_l=\sum_{j\in J_0'}\beta_j\tau_j'e_j+\sum_{j\in J_+'}\beta_je_j-\sum_{j\in J_-'}\beta_je_j+\sum_{l\in L'}\xi_l\mu_l'\end{equation}
  and some $x\in D_{J_0',J_+',J_-',L'}\cap \ri D_{J_0,J_+,J_-,L}$ implying $J_0'\subseteq J_0$ and $L'\subseteq L$, where equality can not simultaneously hold in both inclusions.
  Choosing $x'\in\ri D_{J_0',J_+',J_-',L'}$ and setting $u=x'-x$, we obtain
  \[u_j=0,\ j\in J_0',\ u_j>0,\ j\in (J_0\setminus J_0')\cap J_+',\ u_j<0,\ j\in (J_0\setminus J_0')\cap J_-',\ \skalp{\xi_l,u}=0,\ l\in L', \skalp{\xi_l,u}<0,\ l\in L\setminus L'.\]
  Rearranging \eqref{EqAuxSumMu} yields
  \[\sum_{j\in (J_0\setminus J_0')\cap J_+'}\beta_j(\tau_j-1)e_j+\sum_{j\in (J_0\setminus J_0')\cap J_-'}\beta_j(\xi_j+1)e_j+\sum_{l\in L\setminus L'}\mu_l\xi_l=\sum_{j\in J_0'}\beta_j(\tau_j'-\tau_j)e_j+\sum_{l\in L'}(\mu_l'-\mu_l)\xi_l\]
  and by multiplying this equation by $u$ we obtain the contradiction
  \begin{align*}0&>\sum_{j\in (J_0\setminus J_0')\cap J_+'}\beta_j(\tau_j-1)u_j+\sum_{j\in (J_0\setminus J_0')\cap J_-'}\beta_j(\tau_j+1)u_j+\sum_{l\in L\setminus L'}\mu_l\skalp{\xi_l,u}\\
  &=\sum_{j\in J_0'}\beta_j(\tau_j'-\tau_j)u_j+\sum_{l\in L'}(\mu_l'-\mu_l)\skalp{\xi_l,u}=0.\end{align*}
  Hence, our claim holds true and we may conclude that for every $(J_0,J_+,J_-,L)\in\I$ and every $(z,z^*)\in \ri  E_{J_0,J_+,J_-,L}$ we have
  \begin{align*}\lefteqn{T_{\gph \partial g}(z,z^*)=T_{E_{J_0,J_+,J_-,L}}(z,z^*)=T_{D_{J_0,J_+,J_-,L}}(z)\times T_{\tilde D_{J_0,J_+,J_-,L}}(z^*)}\\
  &=\{w\mv w_j=0, j\in J_0, \skalp{\xi_l,w}=0, l\in L\}\times
  \{\sum_{j\in J_0}\beta_j\sigma_je_j+\sum_{l\in L}\xi_l\nu_l\mv \sigma_j\in\R,\ j\in J_0,\ \nu_l\in\R,\ l\in L\}\\
  &=  W(z)\times W(z)^\perp,\end{align*}
  where the last equality follows from $J_0=J_0(z)$ and $L=L(z)$.
  Now consider $(x,x^*)\in\gph \partial g$. Then $(J_0(x),J_+(x),J_-(x),L(x))\in \I$ and $x\in \ri D_{I_0(x),I_+(x),I_-(x),J(x)}$. Selecting $z^*\in \ri \tilde D_{J_0(x),J_+(x),J_-(x),L(x)}$, for all $\alpha\in(0,1]$ we have $x_\alpha^*:= (1-\alpha)x^*+\alpha z^*)\in \ri \tilde D_{J_0(x),J_+(x),J_-(x),L(x)}$ implying $T_{\gph \partial g}(x,x_\alpha^*)= W(x)\times W(x)^\perp$. Now the assertion follows from the definition of $\Sp\partial g(x,x^*)$ together with Theorem \ref{ThSubdiffConv}.
\end{proof}
}\fi
Let ${\hat d^*}{}^{(k)}:=-\gamma^{(k)}u^{(k)}-f(x^{(k)})$. By Lemma \ref{LemSepStruct} and consecutive application of Theorem \ref{ThSubSpOligo} with $g=q^i$ we obtain
\[\prod_{i=1}^n (W^i)^{(k)}\times \prod_{i=1}^n {(W^i)^{(k)}}^\perp\in\Sp^*\partial q(\hat d^{(k)},{\hat d^*}{}^{(k)}),\]
where for each $i=1,\ldots,n$ the subspace $(W^i)^{(k)}\subset\R^m$  is given by
\[(W^i)^{(k)}:=\{w\mv \skalp{\xi^i_l,w}=0, l\in (L^i)^{(k)},\ w_j=0,\ j\in (J^i_0)^{(k)}\}\]
with $(J^i_0)^{(k)}:=\{j\in \{1,\ldots,m\}\mv (\hat d^i_j)^{(k)}=a^i_j\}$, $(L^i)^{(k)}:=\{l\in \{1,\ldots, p^i\}\mv \skalp{\xi^i_l, (\hat d^i)^{(k)}}=\zeta^i_l\}$,
and the vectors $\xi^i_l$, $l=1,\ldots, p^i$, given by the $l$-th row of the matrix $\Xi^i$.

The required matrices $Y^{(k)}={\rm diag\,}\big((Y^1)^{(k)},\ldots, (Y^n)^{(k)}\big)$ and $X^{(k)}={\rm diag\,}\big((X^1)^{(k)},\ldots, (X^n)^{(k)}\big)$ are block diagonal matrices, where the diagonal $m\times m$ blocks can be computed as
\[(Y^i)^{(k)}=Q_2^i\times {Q_2^i}^T,\quad (X^i)^{(k)}=Q_1^i\times {Q_ 1^i}^T\]
and the columns of $Q_2^i$ and $Q_1^i$ are orthonormal bases for the subspaces $(W^i)^{(k)}$ and ${(W^i)^{(k)}}^\perp$, respectively. The matrices $Q_1^i$ and $Q_2^i$ can be computed, e.g., via a QR-factorization with column pivoting for the matrix with columns $\xi^i_l/\norm{\xi^i_l}$, $l\in (L^i)^{(k)}$, and $e_j$, $j\in (J^i_0)^{(k)}$, see, e.g., \cite[Section 2.2.5.3]{GiMuWr81}.

Concerning the numerical tests\footnote{All codes can be found on \url{https://www.numa.uni-linz.ac.at/~gfrerer/Software/Cournot_Nash/}}, we consider first an academic example with $n=5$ and $m=3$.
The parameters $ b_j^i, \delta_j^i, K_j^i$ of production cost functions together with the market elasticities $\gamma_j$ arising in the inverse demand functions are displayed in Table 1. In the constraints  $\Xi^ix^i\leq \zeta^i$, defining the sets of feasible productions, we
assume that matrices $\Xi^i$ have only one row (i.e., $p^i=1 $). The respective data are listed in Table 2 together with the weights $\beta_j^i$ specifying the costs of change and the "previous" productions $a^i_j$. Finally, Table 3 presents the starting values of $x^i_j$ (initial iteration) and the obtained results, including both the equilibrium productions as well as the corresponding costs of change.

\newcolumntype{d}{>{\hsize=1\hsize}X}
\newcolumntype{s}{>{\hsize=.5\hsize}X}
\newcolumntype{L}{>{\raggedright\arraybackslash}X}
\newcolumntype{R}{>{\raggedleft\arraybackslash}X}
\newcolumntype{C}{>{\centering\arraybackslash}X}
\newcolumntype{Z}{>{\hsize=.875\hsize\RaggedLeft}X}

\begin{table}[H]
\normalsize
\begin{tabularx}{\textwidth}{R | R R R | R R R | R R R | R R R | R R R | R }
  & \multicolumn{3}{c|}{$i=1$}
  & \multicolumn{3}{c|}{$i=2$}
  & \multicolumn{3}{c|}{$i=3$}
  & \multicolumn{3}{c|}{$i=4$}
  & \multicolumn{3}{c|}{$i=5$}
  &\\
\hline
 & $b_j^i$ & $\delta_j^i$ & $K_j^i$ & $b_j^i$ & $\delta_j^i$ & $K_j^i$ & $b_j^i$ & $\delta_j^i$ & $K_j^i$ & $b_j^i$ & $\delta_j^i$ & $K_j^i$ & $b_j^i$ & $\delta_j^i$ & $K_j^i$ & $\gamma_j$ \\
\hline
j=1&9.0 &1.2 &5.0 &7.0 &1.1 &5.0 &3.0 &1.0 &5.0 &4.0 &0.9 &5.0 &2.0 &0.8 &5.0 &1.0\\
j=2 &9.0 &1.2 &5.0 &7.0 &1.1 &5.0 &3.0 &1.0 &5.0 &4.0 &0.9 &5.0 &2.0 &0.8 &5.0 &0.9\\
j=3 &9.0 &1.2 &5.0 &7.0 &1.1 &5.0 &3.0 &1.0 &5.0 &4.0 &0.9 &5.0 &2.0 &0.8 &5.0 &0.8\\
\hline
\end{tabularx}
\caption{Input parameters $b_j^i, \delta_j^i, K_j^i$ of production costs and market elasticities $\gamma_j$.}  
\end{table}
\begin{table}[H]
\begin{tabularx}{\textwidth}{R | R R R | R R R | R R R | R R R | R R R  }
  & \multicolumn{3}{c|}{$i=1$}
  & \multicolumn{3}{c|}{$i=2$}
  & \multicolumn{3}{c|}{$i=3$}
  & \multicolumn{3}{c|}{$i=4$}
  & \multicolumn{3}{c}{$i=5$}\\
\hline
 & $\Xi_j^i$ & $\beta_j^i$ & $a_j^i$ & $\Xi_j^i$ & $\beta_j^i$ & $a_j^i$ & $\Xi_j^i$ & $\beta_j^i$ & $a_j^i$ & $\Xi_j^i$ & $\beta_j^i$ & $a_j^i$ & $\Xi_j^i$ & $\beta_j^i$ & $a_j^i$ \\
\hline
j=1&1.0 &0.5 &47.8 &1.0 &1.0 &51.1 &1.0 &2.0 &51.3 &1.0 &0.0 &48.5 &1.0 &0.0 &43.5\\
j=2 &1.0 &0.5 &47.8 &1.0 &1.0 &51.1 &1.0 &2.0 &51.3 &1.0 &0.0 &48.5 &1.0 &0.0 &43.5\\
j=3 &1.0 &20.0 &47.8 &1.0 &1.0 &51.1 &1.0 &2.0 &51.3 &1.0 &0.0 &48.5 &1.0 &0.0 &43.5\\
\hline
$\zeta^i$ & 200 & & &  250 & & &   100 & & &  200 & & &  200
\end{tabularx}
\caption{Input parameters $\Xi_j^i, \zeta^i$ defining feasible productions, parameters $\beta_j^i$ of costs of change and previous
productions $a_j^i$.}  
\end{table}
\begin{table}[H]
\begin{tabularx}{\textwidth}{R | R R Z | R R R | R R R | R R R | R R R  }
  & \multicolumn{3}{c|}{$i=1$}
  & \multicolumn{3}{c|}{$i=2$}
  & \multicolumn{3}{c|}{$i=3$}
  & \multicolumn{3}{c|}{$i=4$}
  & \multicolumn{3}{c}{$i=5$}\\
\hline
 & $(x^i_j)^{(0)}$ & $x^i_j$ & $z^i_j$& $(x^i_j)^{(0)}$ & $x^i_j$ & $z^i_j$ & $(x^i_j)^{(0)}$ & $x^i_j$ & $z^i_j$ & $(x^i_j)^{(0)}$ & $x^i_j$ & $z^i_j$ & $(x^i_j)^{(0)}$ & $x^i_j$ & $z^i_j$ \\
\hline
j=1&45.0 &54.4 &3.3 &45.0 &54.6 &3.5 &45.0 &20.6 &61.4 &45.0 &50.8 &0.0 &45.0 &45.3 &0.0\\
j=2 &45.0 &67.9 &10.0 &45.0 &66.2 &15.0 &45.0 &30.6 &41.5 &45.0 &58.2 &0.0 &45.0 &50.6 &0.0\\
j=3 &45.0 &47.8 &0.0 &45.0 &85.0 &33.8 &45.0 &48.8 &5.0 &45.0 &70.7 &0.0 &45.0 &60.0 &0.0
\end{tabularx}
\caption{Initial productions $(x^i_j)^{(0)}$, the computed equilibrium productions $x^i_j$ and the corresponding costs of change denoted by $z^i_j$.}  
\end{table}

\if{
\begin{figure}[H]
\centering
\includegraphics[width=0.65\textwidth]{convergence.pdf}
\vspace{-0.5cm}
\caption{Convergence of residuum $|| \hat u^{(k)}||$ .}\label{example_2D}
\end{figure}
}\fi

The results displayed in Table 3 have been achieved in 6 iterations of Algorithm \ref{AlgSSNewtVIBasic} and the final residual amounts to $2.7\times 10^{-12}$.
\if{and the corresponding value of $|| \hat u^{(k)}||$ amounts to  1e-14, see Figure
\ref{example_2D}.}\fi
Note that the third firm exhausts its maximum production capacity whereas the other firms do not. We also observe the prohibitive influence of the high value of $\beta_3^1$, thanks to which, expectantly,
$x_3^1 = a_3^1 = 47.8$.

Next, to demonstrate the computational efficiency of the SCD
\ssstar Newton  method, we increase substantially the values of $n$ and $m$. In dependence of $n$ and $m$, we generated test problems by drawing the data  independently from the uniform distributions with the following parameters:
\[\left.\begin{array}{l}
  b_j^i\sim\U(2,20),\ \delta_j^i\sim\U(0.5,2),\ K_j^i\sim\U(0.1,10)\\
  \Xi^i_{lj}\sim\U(0,1),\ \beta_j^i\sim\U(1,10),\ a_j^i\sim\U(20,50),\\
  \gamma_j\sim\U(1,2)
\end{array}\right\},\quad i=1,\ldots,n, j=1,\ldots,m, l=1,\ldots,p^i\]
Here the numbers $p^i$, $i=1,\ldots,n$ are obtained by rounding  numbers drawn independently from $\U(1,1.5m+1)$. Further we set $\zeta^i:=\Xi^iz^i$, where for each $i=1,\ldots,n$ the elements $z_j^i$, $j=1,\ldots,m$ are drawn from $\U(1,15)$.
For each pair $(n,m)$ belonging to the set $\{(5,200),(25,40), (200,5)\}$ we generated 50 test problems and solved them as well with the heuristic from Subsection \ref{SecHeuristic} as with the globalized \ssstar Newton method  of Algorithm \ref{AlgSSNewt_Hybrid} with $\T=\T_{\gamma}^{\rm PM}$. As a stopping criterion we used $r_{\gamma^{(k)}}(x^{(k)})\leq 10^{-12}r_{\gamma^{(0)}}(x^{(0)})$ and as the starting point we chose the vector $(5,5,\ldots,5)$. Both methods succeeded in all of the $150$ test problems.
 In  Table \ref{TabMeanIter} we report for each scenario the mean value of the iterations needed, the standard deviation and the maximum iteration number.
 \begin{table}[H]
 \begin{tabular}{|c||c|c|c||c|c|c|}
 \hline
 &\multicolumn{3}{c||}{Hybrid method}&\multicolumn{3}{c|}{Heuristic}\\
 \hline
 $(n,m)$&mean value&std. dev.&max. iteration \#&mean value&std. dev.&max. iteration \#\\
 \hline
 (5,200)&20.2&9.6&46&20.4&4.7&39\\
 (25,40)&28.9&10.3&52&28.2&7.6&50\\
 (200,5)&32.4&13.8&76&27.9&9.3&75\\
 \hline
 \end{tabular}
 \caption{\label{TabMeanIter}Statistics of iteration numbers for 50 test problems per scenario}
 \end{table}
For each of the 3 scenarios we have a problem with $nm=1000$ unknowns. The time consuming parts of the \ssstar Newton method are approximation step and the Newton step: In the approximation step we have to solve $n$ quadratic problems with $2m$ variables, whereas in the Newton step we must solve a linear system in $nm$ variables. Thus, in case when $(n,m)=(5,200)$ the approximation step is more time consuming than the Newton step, whereas in case when $(n,m)=(200,5)$ the approximation step is much cheaper than the Newton step. We can see that the iteration numbers needed are fairly small.
Note that the given iteration numbers essentially reflect the {\em global} convergence behaviour: The majority of the iterations is needed to come sufficiently close to the solution and then, by superlinear convergence of the \ssstar Newton method, only 3--6 iterations more are required to approximate the solution with the desired accuracy. In Figure \ref{FigGlobConv1} we depict the residuals  $r_{\gamma^{(k)}}(x^{(k)})$ given by \eqref{EqRes1} for one test problem with $(n,m)=(5,200)$ for both Algorithm \ref{AlgSSNewt_Hybrid} and the heuristic of Subsection \ref{SecHeuristic}.  Algorithm \ref{AlgSSNewt_Hybrid} needed 16 iterations to reduce the initial residual of $840.7$ to $6.0$  and the method stopped after 6 additional iterations with a residual of $4\times10^{-12}$. Similarly, for the heuristic we obtained at the 15-th iterate a residual of $8.5$ and the method stopped after 21 iterations with a final residual of $5.7\times10^{-12}$.

\begin{figure}\centering
  \includegraphics[height=5cm]{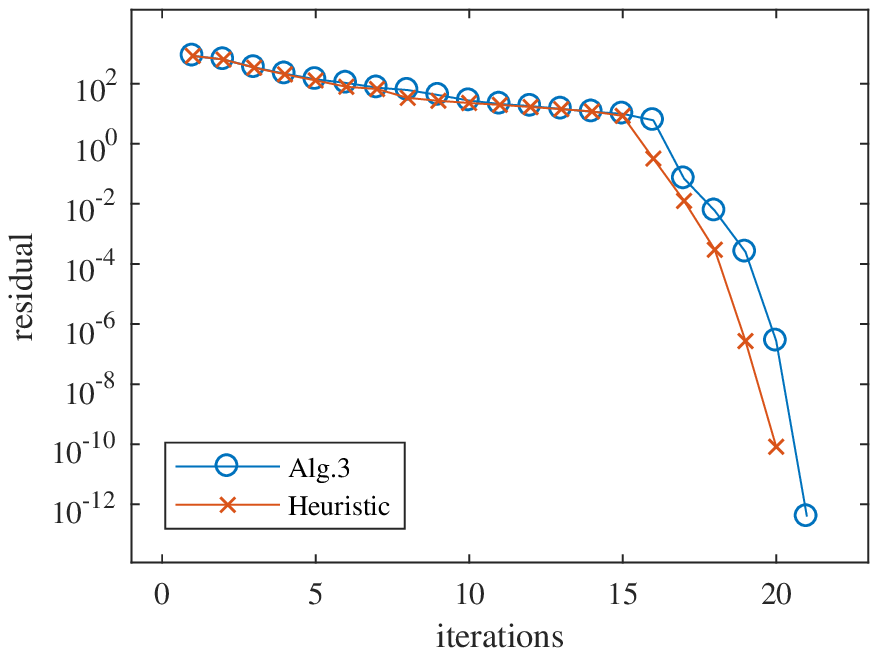}
  \caption{\label{FigGlobConv1}Comparison of Algorithm \ref{AlgSSNewt_Hybrid} with heuristic}
\end{figure}

We now compare the \ssstar Newton method with several first-order splitting method, namely the  Forward-Backward splitting method FB, the golden ratio algorithm aGRAAL \cite{Mal20}, the Douglas-Rachford splitting algorithm DR and the hybrid projection-proximal point algorithm PM \cite{SolSv99}. We performed this comparison only for the scenario with $(n,m)=(200,5)$, where one evaluation of the proximal mapping is relatively cheap, i.e., we have to solve 200 quadratic programs with 10 variables.
We generated 3 test problems  and computed with the \ssstar Newton method a fairly accurate  approximation $\tilde x$ of the exact solution: For each of the 3 test problems the final residual was less than $2.4\times 10^{-12}$. Using this approximate solution $\tilde x$, we computed for the aforementioned  methods the relative error of the iterates $x^{(k)}$ defined as  $\max\{\frac{\vert x_i^{(k)}-\tilde x_i\vert}{\max\{1,\vert \tilde x_i\vert\}}\mv i=1,\ldots,nm\}$. In Figure \ref{FigComparison} we plot  this relative error against the CPU-time needed for calculating $x^{(k)}$. We set for the first-order methods as a time limit five times the time needed for the \ssstar Newton method to converge.
\begin{figure}\centering
  \includegraphics[width=5cm]{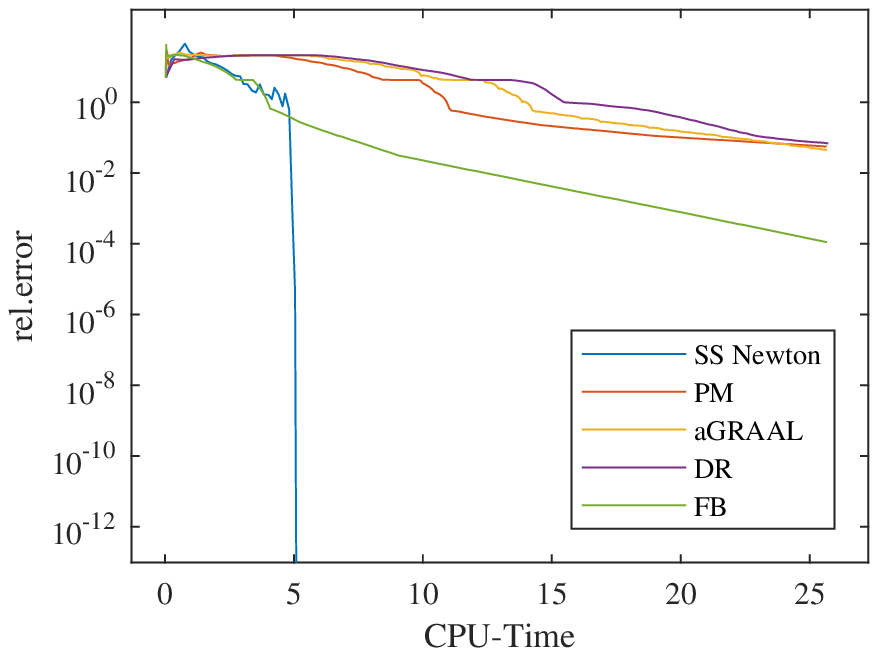}
  \includegraphics[width=5cm]{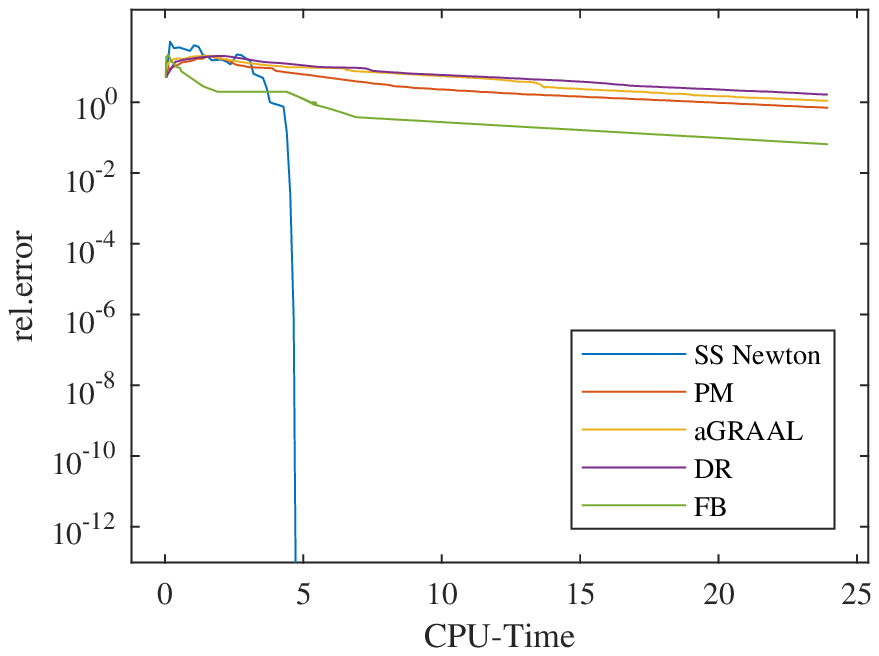}
  \includegraphics[width=5cm]{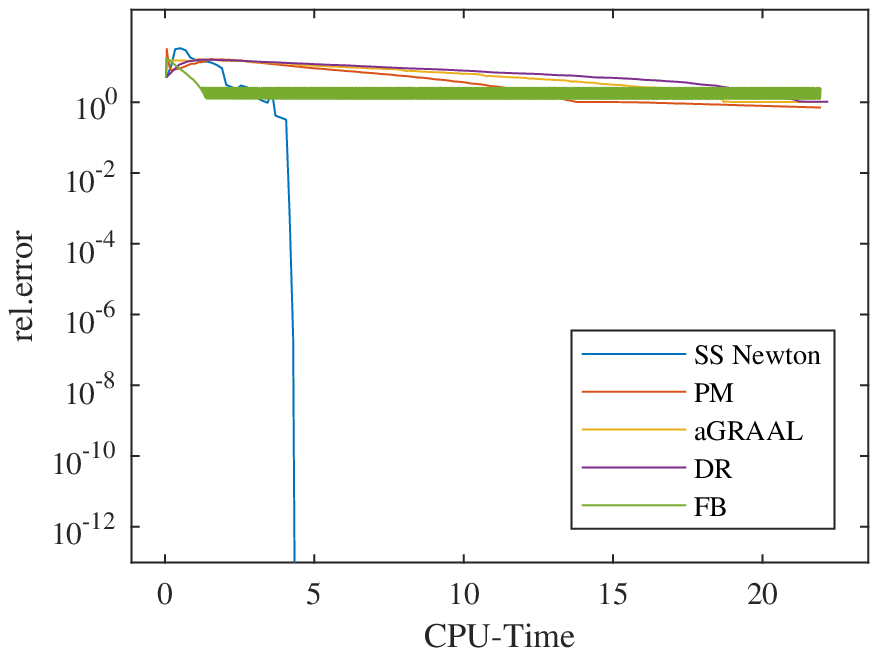}
  \caption{\label{FigComparison}Comparison of the \ssstar Newton method with several first-order methods}
\end{figure}
We can see that only for the first test problem the FB method was able to produce an approximate solution with high accuracy within the time limit. For the FB method, the final relative error was less than $10^{-5}$,  the other methods terminated with a relative error in the range between $4\%$ and  $7\%$. For the second test problem, the relative accuracy of the final iterate for the FB-method was about $8\%$, whereas we could not get even one significant digit with the other methods. For the third test problem, the relative error was for all first-order methods about $100\%$.

\section{Conclusion}
The \ssstar Newton method from \cite{GfrOut21} and its SCD variant from \cite{GfrOut21a} provide us with a powerful tool for numerical solution of a broad class of problems governed by GEs. When facing a concrete problem of this sort, one has to employ appropriate results of variational analysis
in order to implement the AS and the NS in an efficient way. In this paper we suggest an implementation of the SCD \ssstar Newton method for the case of variational inequalities of the 2nd kind, which is a useful modelling framework for a number of practical problems. In particular, in this way one can
model Nash games with convex, possibly nonsmooth  costs, frequently arising, e.g., in economics and biology. Without substantial changes this implementation can be adopted also to the case of the so-called hemivariational inequalities, cf. \cite{Pana}, which are frequently used in various models in nonsmooth mechanics. This could be a topic for a future research.

\section*{Acknowledgements}
The research of the first author was  supported by the Austrian Science Fund
(FWF) under grant P29190-N32. The  research of the second author was supported by the Grant Agency of the
Czech Republic, Project 21-06569K, and the Australian Research Council, Project DP160100854. The  research of the third author was supported by the Grant Agency of the
Czech Republic, Project 21-06569K.

\end{document}